\documentclass[12pt]{article}
\pdfoutput=1

\usepackage[margin=1in]{geometry}

\usepackage{amsmath,amssymb,amsthm}
\usepackage{mathrsfs}

\usepackage{graphicx}
\usepackage{subcaption}
\usepackage{float}

\usepackage[dvipsnames]{xcolor}

\usepackage[colorlinks=true, allcolors=black]{hyperref}

\usepackage[authoryear, round]{natbib}

\theoremstyle{plain}
\newtheorem{theorem}{Theorem}[section]
\newtheorem{lemma}[theorem]{Lemma}
\newtheorem{proposition}[theorem]{Proposition}
\newtheorem{corollary}[theorem]{Corollary}
\theoremstyle{definition}
\newtheorem{remark}[theorem]{Remark}
\newtheorem{assumption}{Assumption}

\usepackage{setspace}
\usepackage{microtype}
\allowdisplaybreaks[4]

\newcommand{\R}{\mathbb{R}}

\renewcommand{\epsilon}{\varepsilon}
\renewcommand{\hat}{\widehat}
\renewcommand{\tilde}{\widetilde}

\newcommand{\E}{\mathop{\mathbb{E}}\nolimits}

\let\liminf\relax
\newcommand{\liminf}{\mathop{\underline{\mathrm{lim}}}\limits}
\let\limsup\relax
\newcommand{\limsup}{\mathop{\overline{\mathrm{lim}}}\limits}

\DeclareMathOperator*{\argmin}{arg\,min}

\def\f{\frac}
\newcommand{\calX}{\mathcal{X}}
\DeclareMathOperator{\Var}{Var}
\newcommand{\ftil}{\widetilde f}
\def\bff{\mathbf{f}}

\def\veps{\varepsilon}
\def\N{{\mathbb{N}}}
\def\P{{\mathbb{P}}}

\def\scrF{{\mathscr F}}
\def\muhat{\widehat{\mu}}
\def\fhat{\widehat{f}}
\def\uhat{\widehat{u}}
\def\vhat{\widehat{v}}

\def\Ybar{{\overline Y}}

\def\vepsbar{{\overline\veps}}
\def\Z{{\mathbb{Z}}}
\def\L{{\mathbb{L}}}
\def\U{{\mathbb{U}}}
\newcommand{\ind}{\mathbf{1}}
\def\Ytil{\widetilde{Y}}

\title{Statistical Inference for Additive Monotone Models\\
       under the Fixed Lattice Design}

\author{%
  \textbf{Huachen Ren}\\[2pt]
  Department of Statistics,\\
  Rutgers, the State University of New Jersey\\[2pt]
  \texttt{hr306@scarletmail.rutgers.edu}
}

\date{}

\begin{document}

\maketitle

\begin{abstract}
We study statistical inference for least squares estimators (LSEs)
in additive monotone models under a general fixed lattice design.
We establish joint limiting distributions for the LSEs and show that the estimators
of different additive components are asymptotically independent.
The form of the limiting distribution of each component is
determined by how fast the number of design points along the corresponding
coordinate grows relative to the total sample size~\(n\).
Apart from this growth rate, the limit depends only on the noise level
and, in the non-Gaussian regimes, on the local derivative of the component.
In particular, the limit does not depend on the dimension of the model.
When the number of design
points along a coordinate grows faster than \(n^{1/3}\), we construct
tuning-free pointwise confidence intervals based on a pivotal limiting
distribution, and we validate the theory in numerical simulations.
We further show that the block-size normalization underlying these
intervals fails to be pivotal at the critical growth rate \(n^{1/3}\).
We also prove a switching lemma, of independent interest, that simplifies
the derivation of limiting distributions in isotonic regression.
\end{abstract}

\noindent\textbf{Keywords:} isotonic regression, additive model, shape constraint,
fixed lattice design, limiting distribution, confidence interval, Chernoff distribution.

\bigskip

\def\rhat{\widehat r}
\def\ellhat{\widehat\ell}

\section{Introduction}
Nonparametric regression is widely used in statistical applications to model complex relationships between a response $Y$ and covariates $X$, without imposing a predetermined functional form. Classical estimators based on splines and kernels \citep{tsybakov2009nonparametric,wahba1990spline} typically involve tuning parameters, which can present practical challenges. Moreover, when the covariates are multi-dimensional, these estimators often suffer from the curse of dimensionality. Nonparametric additive models provide a popular multivariate extension of univariate nonparametric regression and have been studied extensively; see \citet{Hastie1990}. The additive structure separates covariate effects into component functions, which makes the model easy to interpret. Allowing each component function to have its own smoothness level makes the model flexible in practice. In the seminal work of \citet{stone1985additive}, it was shown that the additive components can be estimated at the same optimal rates as in the univariate setting, thereby avoiding the curse of dimensionality. However, spline- and kernel-based estimators for additive models typically still require hyperparameter tuning.

In many applications, researchers have prior qualitative knowledge such as monotonicity or convexity relationships between $Y$ and components of $X$, for example in biology \citep{obozinski2008consistent}, psychology \citep{kruskal1964multidimensional}, and economics \citep{matzkin1991semiparametric}. Under such shape constraints (monotonicity, convexity, unimodality, etc.), least squares and/or maximum likelihood estimation yields tuning-parameter-free estimators with optimal convergence rates \citep{zhang2002risk,guntuboyina2018nonparametric}, as well as tuning-parameter-free inference procedures \citep{deng2021confidence}.

In modern applications, data are often high-dimensional and come with domain knowledge. Motivated by this, we study the least squares estimator (LSE) for an additive model under monotonicity constraints, which combines the advantages of additive modeling \citep{hastie1986generalized} and isotonic regression \citep{ayer1955empirical, eeden1957maximum, brunk1955maximum}. Given observations $\{(X_i, Y_i)\}_{i=1}^n$ with $X_i=(x_{i,1},\ldots,x_{i,d})^T\in\R^d$ and $Y_i\in\R$, the $d$-dimensional additive monotone model is
\begin{align}\label{eqn-model}
	Y_i &=\mu^* + f_1^*(x_{i,1})+{\cdots}+f_d^*(x_{i,d})+\veps_i,
\end{align}
where $\mu^* \in \R$ is an unknown intercept and $f_1^*,\ldots,f_d^*$ belong to the class $\scrF$ of nondecreasing functions from $[0,1]$ to $\R$. For identifiability, we impose $\sum_{i=1}^n f^*_j(x_{i,j})=0$ for $j=1,\ldots,d$. The errors $\veps_1,\ldots,\veps_n$ are i.i.d.\ with mean zero and variance $\sigma^2$. Since shape-constrained function classes are convex, the least squares estimator is well defined as the projection of the response vector onto the corresponding convex set, and requires no tuning parameters. The LSEs $(\muhat,\fhat_1,\ldots,\fhat_d)$ are defined as any minimizer of
\begin{align}
	(\muhat, \fhat_1, \ldots, \fhat_d) \in \argmin_{(\mu, f_1,\ldots, f_d)\in \R \times \scrF \times \cdots \times \scrF}\sum_{i=1}^n\big(Y_i-\mu - f_1(x_{i,1})-\cdots-f_d(x_{i,d})\big)^2.\label{LSE}
\end{align}
By the convexity of the function class and the projection theory of Hilbert space, $\muhat+\fhat_1+\cdots+\fhat_d$ is uniquely defined on the design points. For the identifiability issue, we further assume that $\sum_{i=1}^n\fhat_j(x_{i,j})=0$ for $j=1,{\ldots}, d$.

\subsection{Related work}
Since the additive monotone model was first proposed by \citet{bacchetti1989additive}, it has been applied in a range of areas \citep{morton2000additive, de2002statistical}. \citet{bacchetti1989additive} proposed a cyclical pool-adjacent-violators algorithm (cyclical PAVA) that combines backfitting \citep{Hastie1990} with the pool-adjacent-violators algorithm (PAVA) \citep{ayer1955empirical} to compute least squares estimators for additive monotone models. \citet{mammen2007additive} showed that, under random design, the componentwise LSE computed by cyclical PAVA is asymptotically equivalent to the LSE in the corresponding univariate problem, which is often referred to as an \emph{oracle property}. Oracle properties for general additive models were studied in \citet{mammen1999existence}. \citet{fang2012lasso} studied additive monotone models from a high-dimensional perspective and developed modified backfitting algorithms for high-dimensional settings.

\citet{meyer2013simple, meyer2013semi} developed a ``hinge'' algorithm for quadratic programming, applied it to additive shape-constrained models, and studied identifiability, statistical inference, and degrees of freedom using the cone structure of these models. Degrees of freedom and model selection for additive monotone models were also studied in \citet{rueda2013degrees}. \citet{chen2016generalized} developed an active-set algorithm for generalized additive shape-constrained index models and studied uniform consistency of the LSE on compact sub-intervals. \citet{han2018robustness} proved that LSEs for additive models with shape constraints are robust to heavy-tailed noise and model misspecification under certain entropy conditions. Rates of convergence and asymptotic distributions for LSEs in partially linear monotone models were studied in \citet{huang2002note, cheng2009semiparametric, cheng2012empirical, yu2014partial}.

Since the LSEs of shape-constrained models are typically not smooth, shape-constrained regression splines  \citep{ramsay1988monotone, meyer2008inference} were applied to construct smooth estimators for additive shape-constrained models. \citet{tutz2007generalized} combined I-splines  \citep{ramsay1988monotone} and boosting to study generalized additive monotone models. \citet{meyer2011bayesian} studied generalized partial linear regression under shape and smoothness constraints from a Bayesian perspective.
\citet{pya2015shape} proposed a framework for modeling generalized additive models with shape constraints based on penalized splines. See \citet{meyer2018framework} and \citet{engebretsen2019additive} for more detailed reference lists on additive shape-constrained models.

Despite the literature above, limiting distributions and statistical inference for LSEs in additive monotone models under fixed lattice designs have not been established, and this paper addresses this gap. Our main results build on a key observation of \citet{guntuboyina2018nonparametric}, which shows that the LSE under a fixed lattice design is equivalent to the LSE in a univariate isotonic regression problem with a triangular array of errors. The local asymptotic distribution for univariate isotonic regression has been studied extensively \citep{brunk1970estimation, wright1981asymptotic, leurgans1982asymptotic, anevski2006general, durot2018limit}. A recent paper \citet{mallick2023new} studies the limiting distribution of univariate isotonic regression with triangular-array errors, allowing the target function to depend on sample size and removing the independence requirement between covariates and errors. However, it is unclear whether existing results apply directly in our setting, especially when the numbers of design points are imbalanced across dimensions and can have different orders of magnitude relative to the local convergence rates.

\subsection{Our contributions}

Our first contribution is a new method for establishing the limiting distribution of $\fhat_j$, which differs from both the ``direct approach'' of \citet{rao1969estimation} and the ``inverse approach'' of \citet{groeneboom1985estimating}. The direct approach typically requires proving that a certain min--max operator is a continuous functional of the underlying stochastic process \citep{anevski2006general, durot2018limit}, which can be nontrivial for general processes \citep{anevski2006general} and in multivariate settings \citep{han2020limit}. Our key step converts the min--max operator into the sum of two maxima over two independent partial-sum processes, which allows us to establish convergence of the distribution function of the LSE. This reduces the problem to studying maxima of partial-sum processes rather than a nested min--max functional. The exact statement appears in Lemma~\ref{lm-switch}, which also helps bypass the small-deviation arguments used in \citet{han2020limit}.

Building on Lemma~\ref{lm-switch}, we establish limiting distributions for LSEs in additive monotone models under a fixed lattice design. Depending on how many design points lie along each coordinate, three asymptotic regimes arise. Suppose there are $n^{\beta_j}$ design points along the $j$th coordinate, where $\beta_j>0$ controls the sampling resolution in that dimension. Let $x_0=(x_{0,1},\ldots,x_{0,d})^T\in(0,1)^d$ be a design point, and let $f_j^{*\prime}(x_{0,j})$ denote the derivative of $f_j^*$ at $x_{0,j}$. Then the LSE $\fhat_j$ satisfies
\begin{align}\label{eqn-0}
    w_{n,j}^{-1}\left(\fhat_j(x_{0,j}) - f_j^*(x_{0,j})\right) \rightsquigarrow 
     \mathbb{D}_{\beta_j, f_j^{*\prime}(x_{0,j})}, 
\end{align}
where $w_{n,j}$ is the local rate of convergence of $\fhat_j$, and
$\mathbb{D}_{\beta_j, f_j^{*\prime}(x_{0,j})}$ is the limiting random variable
defined in Section~\ref{limit-them-section}. Its distribution depends on
$\beta_j$, $f_j^{*\prime}(x_{0,j})$, and $\sigma$. Furthermore, the vector of
normalized estimation errors converges jointly to a random vector with mutually
independent components. When $\beta_j>1/3$, $\mathbb{D}_{\beta_j, f_j^{*\prime}(x_{0,j})}$ coincides
with the limiting distribution of standard univariate isotonic regression. When $\beta_j=1/3$, the limiting distribution is the left derivative of the greatest convex minorant of a two-sided random walk plus a quadratic drift. The greatest convex minorant of a function $f$ is the largest convex function smaller than $f$. The limiting distribution becomes a Gaussian distribution for $\beta_j<1/3$. The asymptotic distribution of $\fhat_j(x_{0,j})$ does not depend on the dimension $d$. As far as we know, this is the first result for the limiting distributions of LSEs of the additive monotone model under the general fixed lattice design.

Our second contribution is a valid inference procedure for the LSEs in additive monotone models. When $\beta_j>1/3$, the limiting distribution in \eqref{eqn-0} depends on the unknown derivative $f_j^{*\prime}(x_{0,j})$, so \eqref{eqn-0} does not directly yield confidence intervals. Since a key advantage of shape-constrained methods is the absence of tuning parameters, we seek a tuning-free approach to inference that avoids derivative estimation. For testing monotone functions, likelihood ratio tests have been studied extensively \citep{banerjee2001likelihood, banerjee2007likelihood}. More recently, \citet{deng2021confidence} proposed a different tuning-free procedure for monotone functions using information beyond point estimates in multivariate isotonic regression, and these ideas were extended to convexity-constrained models in \citet{deng2022inference}. We follow \citet{deng2021confidence} and prove a pivotal limiting distribution for the additive monotone LSE (Theorem~\ref{inference-thm}). The min--max characterization of the univariate isotonic LSE is classical \citep{barlow1972statistical,robertson1988order,groeneboom2014nonparametric}. The quantity on the left-hand side of \eqref{eqn-kink} below is the oracle isotonic fit $\fhat_j^{or}(x_{0,j})$. By Lemma~3.1 of \citet{guntuboyina2018nonparametric}, $\fhat_j=\fhat_j^{or}-\Ybar$, where $\Ybar=n^{-1}\sum_{i=1}^nY_i$. This common additive shift leaves the fitted constant intervals unchanged. For dimension $j$, let $\vhat_j$ and $\uhat_j$ be the lower and upper endpoints of the largest interval of consecutive design points containing $x_{0,j}$ on which the fitted component $\fhat_j$ is constant. Equivalently,
\begin{align}\label{eqn-kink}
    \min_{u_j \ge x_{0,j}} \max_{v_j \le x_{0,j}} \f{\sum_{i=1}^n Y_i\ind\{ x_{i,j} \in [v_j, u_j] \}}{\sum_{i=1}^n\ind\{ x_{i,j} \in [v_j, u_j] \}} = \f{\sum_{i=1}^n Y_i\ind\{ x_{i,j} \in [\vhat_j, \uhat_j] \}}{\sum_{i=1}^n \ind\{ x_{i,j} \in [\vhat_j, \uhat_j] \}}.
\end{align}
Define the corresponding scaled right and left widths by
\[
    \widehat r_j=\f{\uhat_j-x_{0,j}}{w_{n,j}},
    \qquad
    \widehat\ell_j=\f{x_{0,j}-\vhat_j+n_j^{-1}}{w_{n,j}}.
\]
Then
\[
    n_{\widehat r_j,\widehat\ell_j}^j
    :=\sum_{i=1}^n\ind\!\left\{
      x_{i,j}\in(x_{0,j}-\widehat\ell_jw_{n,j},
                 x_{0,j}+\widehat r_jw_{n,j}]
    \right\}
    =\sum_{i=1}^n\ind\{x_{i,j}\in[\vhat_j,\uhat_j]\}
\]
is the exact number of observations in that fitted interval. Then, for $\beta_j> 1/3$,
\begin{align*}
    \sqrt{n_{\widehat r_j,\widehat\ell_j}^j}(\fhat_j(x_{0,j}) - f_j^*(x_{0,j})) \rightsquigarrow \sigma \L_j,
\end{align*}
where $\sigma$ is the noise level, and $\L_j$ does not depend on $f_j^{*\prime}(x_{0,j})$ and is a special case of the pivotal limiting distribution in \citet{deng2021confidence}. Based on the above limiting distribution, we propose the following pointwise confidence interval for $f_j^*(x_{0,j})$:
\begin{align}\label{confidence-interval-def}
    \mathcal{I}_n(x_{0,j};c_{\delta}, j):= \left[ \fhat_j(x_{0,j}) - \f{c_{\delta}\hat{\sigma}}{\sqrt{n_{\widehat r_j,\widehat\ell_j}^j}}, \fhat_j(x_{0,j}) + \f{c_{\delta}\hat{\sigma}}{\sqrt{n_{\widehat r_j,\widehat\ell_j}^j}} \right],
\end{align}
where $\hat{\sigma}$ is the square root of a consistent estimator of $\sigma^2$, and $c_{\delta}$ is the upper-$\delta$ critical value of $|\L_j|$, so that the interval has asymptotic coverage $1-\delta$.  When $\beta_j=1/3$, the normalization by $\sqrt{n_{\widehat r_j,\widehat\ell_j}^j}$ proposed by \citet{deng2021confidence} is not pivotal. Proposition~\ref{failure-inference} shows that the resulting limiting distribution still depends on the unknown derivative $f_j^{*\prime}(x_{0,j})$, and Section~\ref{simulation-failure} illustrates this failure numerically. Whether some other tuning-free normalization of the LSE is pivotal in the critical regime remains an open problem (see Section~\ref{discussion-sec}).

The rest of this paper is organized as follows. Section \ref{limit-them-section} introduces the problem setting and presents our results on the joint limiting distributions of the LSEs under the fixed lattice design. The pivotal limiting distributions and the construction of confidence intervals are presented in Section~\ref{inference-section}. Section~\ref{simulation-sec} contains simulations that validate our theoretical findings, and Section~\ref{discussion-sec} discusses open problems.
The proofs of our main results in Section \ref{limit-them-section} and Section \ref{inference-section} are deferred to Appendices \ref{sec-proof-lim-them} and \ref{pivot-limit-section}, and Appendix \ref{tech-lemma-sec} collects the technical lemmas.

\subsection{Notation}
For a real-valued function $f$ defined on $\R$, we denote its first-order derivative by $f'$ and its greatest convex minorant by $gcm(f)$, which is the largest convex function lying below $f$. The left derivative of $gcm(f)$ at a point $x$ is written as $gcm(f)'_-(x)$. The cardinality of a set $A$ is denoted by $|A|$. For $k \in \N$ and subsets $A_1,\ldots,A_k \subset \R$, we denote their Cartesian product by $\prod_{i=1}^{k}A_i \subset \R^k$. For two vectors $x \in \R^d$ and $y \in \R^d$, let $x \odot y$ denote the elementwise product of $x$ and $y$, that is, $(x_1y_1,\ldots,x_dy_d)^T$. We use $\rightsquigarrow$ to denote weak convergence of probability distributions.

For a function $f: [0,1] \to \mathbb{R}$ or an array $f$ indexed by $[0,1]$, and any finite subset $A \subset [0,1]$, define the sample mean of $f$ over $A$ as
\begin{align*}
\overline{f}|_{A} = \frac{1}{|A|}\sum_{x \in A}f(x).
\end{align*}

For two real numbers $a,b$, $a\vee b \equiv \max\{a,b\}$ and $a\wedge b\equiv \min\{a,b\}$. $C_x$ will denote a generic finite constant that only depends on a generic quantity $x$, whose numeric value may change from line to line unless otherwise specified. $a\lesssim_x b$ and $a\gtrsim_x b$ mean $a\le C_xb$ and $a\ge C_xb$ respectively, and $a\asymp_x b$ means $a\lesssim_x b$ and $a\gtrsim_x b$ ($a\lesssim b$ means $a\le Cb$ for some absolute constant $C$). For two sequences $(a_n)$ and $(b_n)$ (with $b_n>0$), we write $a_n \ll b_n$ (resp. $a_n \gg b_n$) if $a_n/b_n \to 0$ (resp. $a_n/b_n \to \infty$) as $n\to\infty$. $O_p$ and $o_p$ denote the usual big and small $O$ notations in probability. We write $\Z$, $\Z_{\neq 0}$, $\Z_{\ge 0}$, and $\Z_{\ge 1}$ for the integers, the nonzero integers, the nonnegative integers, and the positive integers, respectively.

\section{Limiting distributions under the fixed lattice design}\label{limit-them-section}
We first state the local smoothness and fixed-design assumptions used to derive the limiting distributions of the LSEs in additive monotone models.

\begin{assumption}\label{smoothness-assumption}{For each $j=1,\ldots,d$, $f_j^*$ is nondecreasing on $[0,1]$ and differentiable at $x_{0,j}$ with $f_j^{*\prime}(x_{0,j})>0$, in the sense that for all $L_0>0$,}
    \begin{align}\label{local-smooth-assumption}
        \lim_{\delta \searrow 0} \delta^{-1} \sup_{\substack{x \in [0,1],\\ |x- x_{0,j}| \le L_0\delta}}|f_j^*(x) - f_j^*(x_{0,j}) - f_j^{*\prime}(x_{0,j})(x - x_{0,j})| = 0.
    \end{align}
\end{assumption}
Next, we state our assumption on the design points $\{X_1, \ldots, X_n\}$.
\begin{assumption}\label{fixed-design-assumption}
The set of design points $\mathcal{X} := \{X_1, \ldots, X_n\} \subset \R^d$ is fixed and is the full Cartesian lattice $\mathcal{X} = \prod_{l=1}^d \mathcal{X}_l$ with $\mathcal{X}_l = \{k/n_l : k = 1, \ldots, n_l\} \subset [0,1]$, and $x_0 \in \mathcal{X}$. In particular, $n = \prod_{l=1}^d n_l$. Asymptotics are taken along a sequence of such designs in which $n_l \to \infty$ for every $l$ and the exponents $\beta_l := \log n_l / \log n$ stay fixed, so that $n_l = n^{\beta_l}$ exactly, with $\beta_l \in (0,1)$ and $\sum_{l=1}^d \beta_l = 1$.
\end{assumption}

\begin{remark}
Fixing the exponents in Assumption~\ref{fixed-design-assumption} requires each $\beta_l$ to be rational: if $\beta_l = p_l/q$ with $\sum_{l=1}^d p_l = q$, the lattices with $n_l = m^{p_l}$ for $m \in \N$ satisfy the assumption. Irrational exponents can be accommodated by requiring only $\log n_l / \log n \to \beta_l$, at the cost of carrying approximation terms through the arguments. We omit these routine modifications.
\end{remark}

We first consider the rate of convergence of $\fhat_j$ and denote it by $w_{n,j}$. Heuristically, when the first-order derivative of $f_j^*$ does not vanish, the least squares estimator of $f_j^*$ can be localized to a strip along the $j$th dimension containing roughly $n^{\beta_j}w_{n,j} \cdot n^{1-\beta_j}$ observations. The bias is of order $w_{n,j}$ and the standard deviation of the strip average is of order $n^{-(1-\beta_j)/2}/\sqrt{n^{\beta_j}w_{n,j}}$. Balancing the bias against the noise, subject to the constraint that the strip contains at least one design point, gives
\begin{align*}
	w_{n,j} \asymp \f{n^{-(1-\beta_j)/2}}{\sqrt{n^{\beta_j}w_{n,j}}}, \qquad n^{\beta_j}w_{n,j} \ge 1,
\end{align*}
which implies
\begin{align}\label{rate-cvgc-def}
    w_{n,j} =
    \begin{cases}
        n^{-(1-\beta_j)/2},  & \beta_j <\f{1}{3},  \\
        n^{-1/3}, & \beta_j \ge 1/3. \\
    \end{cases}
\end{align}
A detailed proof of the rate of convergence can be found in Proposition \ref{rate-cvgc}.

Given a sequence of independent standard normal random variables $\{ Z_i \}_{i \in \Z_{\neq 0}}$, consider the associated two-sided random walk $S$ with $S_0 = 0$,  $S_n = \sum_{i = 1}^n Z_i$  and $S_{-n} = \sum_{i=1}^{n} -Z_{-i}$ for $n \ge 1$ and the random piecewise linear function $t \rightarrow S(t)$, $t \in \R$, obtained by linear interpolations
between the values $S(n):= S_n$ for $n \in \Z$. Define $p_n = \f{n(n+1)}{2}$ for $n \in \Z$ and $p(t)$ as the linear interpolation between the values $p(n):=p_n$ for $n \in \Z$. 
For $\alpha \in \R$, define
\begin{align}
    \mathbb{G}_{\alpha}&:= gcm(\sigma S+\alpha p)'_-(0)\nonumber\\
    &=\inf_{r \in \Z_{\ge 0}} \sup_{\ell \in \Z_{\ge 1} } \left\{\f{ \sigma S(r) - \sigma S(-\ell)}{r+\ell}+\f{\alpha(r-\ell+1)}{2}\right\},\label{gcm-rw-def}
\end{align}
where the equality is by the definition of the left derivative of the greatest convex minorant of a function, see Section 3.3 in \citet{groeneboom2014nonparametric}.

For $\beta \in (0,1)$ and $\alpha > 0$, let $\mathbb{D}_{\beta,\alpha}$ denote a random variable defined by
\begin{align}\label{limit-dist-def}
    \mathbb{D}_{\beta, \alpha}=
    \begin{cases}
        (4\sigma^2\alpha)^{1/3}\mathbb{C},  & \beta >\f{1}{3},  \\
         \sigma Z, & \beta <1/3, \\
         \mathbb{G}_{\alpha}, & \beta = 1/3,
    \end{cases}
\end{align}
where $\mathbb{C}:=\argmin_{h \in \R}\{W(h)+h^2\}$ has the Chernoff distribution, $W$ is a two-sided Brownian motion starting from zero, $Z\sim N(0,1)$, and $\mathbb{G}_{\alpha}$ is defined in \eqref{gcm-rw-def}. The notation suppresses the dependence on $\sigma$.

\begin{theorem}\label{theorem-lattice}
    Suppose Assumptions \ref{smoothness-assumption} and \ref{fixed-design-assumption} hold. Let $D_1,\ldots,D_d$ be mutually independent random variables such that
    $D_j\stackrel{d}{=}\mathbb{D}_{\beta_j,f_j^{*\prime}(x_{0,j})}$ for each $j$. With $w_{n,j}$ and $\mathbb{D}_{\beta, \alpha}$ defined in \eqref{rate-cvgc-def} and \eqref{limit-dist-def}, we have
 \begin{align*}
    \begin{pmatrix} w_{n,1}^{-1}\left(\fhat_1(x_{0,1}) - f_1^*(x_{0,1})\right) \\ \vdots \\ w_{n,d}^{-1}\left(\fhat_d(x_{0,d}) - f_d^*(x_{0,d})\right) \end{pmatrix} \rightsquigarrow 
    \begin{pmatrix} D_1 \\ \vdots \\ D_d \end{pmatrix}.
 \end{align*}
\end{theorem}
The above result establishes weak convergence of the vector of least squares estimators for the additive monotone model. The marginal limiting distribution follows by relating the LSE of each component to the LSE in a univariate isotonic regression problem. We establish asymptotic joint independence using characteristic functions, by controlling the dependence among the componentwise estimators.
\begin{remark}
    We do not have a closed form for the limiting distribution $\mathbb{D}_{\beta, \alpha}$ when $\beta = 1/3$. In Section \ref{simulation-sec}, we simulate the distribution by approximating the two-sided random walk over $(-\infty, \infty)$ by a two-sided random walk on finite intervals. It is an interesting question to see whether the technique in \citet{groeneboom1989brownian} can be applied to find the distribution analytically.
\end{remark}
\begin{remark}\label{remark-random-design}
    Under a random design, the LSE for each component does not decouple from the others as cleanly as it does under a fixed lattice design. Consequently, one must control the contribution of the other components, and the arguments for fixed lattice designs do not carry over directly. While \citet{mammen2007additive} established an ``oracle property'' for random designs, their proof relies on the claim that the LSEs in additive monotone models are uniformly bounded with high probability (see Lemma 3 of \citet{mammen2007additive}). To the best of our knowledge, a complete proof of this claim is still unavailable.
\end{remark}

\section{Statistical inference}\label{inference-section}
In this section, we construct confidence intervals for $f_j^*(x_{0,j})$. By Theorem~\ref{theorem-lattice}, when $\beta_j \ge 1/3$, the limiting distribution of $w_{n,j}^{-1}(\fhat_j(x_{0,j}) - f_j^*(x_{0,j}))$ depends on the unknown derivative $f_j^{*\prime}(x_{0,j})$, whereas for $\beta_j<1/3$ the limiting distribution is Gaussian and depends only on $\sigma$. Since estimating $f_j^{*\prime}(x_{0,j})$ typically requires tuning parameters, and a key appeal of LSEs under shape constraints is their tuning-free nature, we develop a pivotal limiting distribution that yields tuning-free confidence intervals. Building on \citet{deng2021confidence}, we obtain a pivotal limit for $\beta_j>1/3$ using the number of design points in the strip identified by the LSE. For $\beta_j=1/3$, we show that the normalization by $\sqrt{n_{\widehat r_j,\widehat\ell_j}^j}$ is not pivotal (Proposition~\ref{failure-inference}).

\subsection{Pivotal limiting distribution theory}
We first state our pivotal limiting distribution for $\beta_j > 1/3$. Recall that $(\widehat r_j,\widehat\ell_j)$ are the scaled widths of the fitted interval in \eqref{eqn-kink}. For arbitrary scaled right and left widths $r\ge0$ and $\ell>0$, define the number of observations in the corresponding half-open interval by
\begin{align}\label{n_u_v-def}
    n_{r,\ell}^j
    =\sum_{i=1}^n
      \ind\{x_{i,j}\in(x_{0,j}-\ell w_{n,j},x_{0,j}+rw_{n,j}]\}.
\end{align}
Let $W(h)$ be a two-sided Brownian motion starting at zero, then the pivotal distribution is defined by
\begin{align*}
    \L_j &\equiv \sqrt{g_{j}^* + h_j^*}\,\min_{g \ge 0}\max_{h >0}\left\{\f{W(g)- W(-h)}{g + h}+ g-h\right\},
\end{align*}
where  $g_j^*$ and $h_j^*$ are defined as the minimizer and maximizer of the above min-max operation and are almost surely unique  \citep{kim1990cube, deng2021confidence}, i.e.
\begin{align}
    \min_{g \ge 0}\max_{h >0}\left\{\f{W(g)- W(-h)}{g + h}+ g-h\right\} = \f{W(g_{j}^*)- W(-h_j^*)}{g_{j}^* + h_j^*}+ g_{j}^*-h_j^*.
\end{align}

\begin{theorem}\label{inference-thm}
    Under the same assumptions as in Theorem \ref{theorem-lattice}, with $n_{r,\ell}^j$ and $\L_j$ defined above, for $\beta_j>1/3$ we have,
    \begin{align*}
        \sqrt{n_{\widehat r_j,\widehat\ell_j}^j}(\fhat_j(x_{0,j}) - f_j^*(x_{0,j})) \rightsquigarrow \sigma \L_j,
    \end{align*}
\end{theorem}

The above pivotal limiting distribution theory implies the tuning-free confidence interval in \eqref{confidence-interval-def}.
\begin{corollary}
    Let $c_{\delta}$ be a continuity point of the distribution function of $|\L_j|$ such that $\P(|\L_j| > c_{\delta}) = \delta$ and $\hat{\sigma}^2$ be a consistent estimator of $\sigma^2$. Then the CI defined in \eqref{confidence-interval-def} satisfies
    \begin{align*}
        \lim_{n \to \infty} \P\big(f_j^*(x_{0,j}) \in \mathcal{I}_n(x_{0,j};c_{\delta}, j)\big) = 1-\delta.
    \end{align*}
\end{corollary}

\begin{remark}\label{remark-sigma-hat}
    A simple consistent estimator of $\sigma^2$ that requires no tuning parameters is available from the residuals of the LSE: since the fitted regression function $\muhat + \sum_{j=1}^d \fhat_j$ is consistent in the empirical $L_2$ norm \citep{guntuboyina2018nonparametric}, the estimator
    \[
        \hat{\sigma}^2 = \f{1}{n}\sum_{i=1}^n \Big(Y_i - \muhat - \sum_{j=1}^d \fhat_j(x_{i,j})\Big)^2
    \]
    satisfies $\hat{\sigma}^2 \stackrel{p}{\to} \sigma^2$. Alternatively, difference-based estimators along any single coordinate of the lattice \citep{rice1984bandwidth} can be used. Either choice preserves the tuning-free nature of the procedure.
\end{remark}

The main idea is that $n_{\widehat r_j,\widehat\ell_j}^j$ contains information about $f_j^{*\prime}(x_{0,j})$, and pivotality follows from joint weak convergence of the estimator and the scaled block widths. At the critical resolution $\beta_j=1/3$, normalization by $\sqrt{n_{\widehat r_j,\widehat\ell_j}^j}$ has a non-pivotal limit. For $a>0$, recall $\mathbb{G}_a$ from \eqref{gcm-rw-def}. Let $r_a^*$ and $\ell_a^*$ denote its almost surely unique minimizing and maximizing arguments, respectively.

\begin{proposition}\label{failure-inference}
    Under the same assumptions as in Theorem \ref{theorem-lattice}, let
    $a_j=f_j^{*\prime}(x_{0,j})$ and suppose $\beta_j=1/3$. Then
    \begin{align*}
        \sqrt{n_{\widehat r_j,\widehat\ell_j}^j}
        \left(\fhat_j(x_{0,j})-f_j^*(x_{0,j})\right)
        \rightsquigarrow
        \L_j^{\mathrm{crit}}(a_j)
        :=
        \sqrt{r_{a_j}^*+\ell_{a_j}^*}\,\mathbb{G}_{a_j}.
    \end{align*}
    The distribution of $\L_j^{\mathrm{crit}}(a_j)/\sigma$ depends on $a_j$, and hence normalization by $\sqrt{n_{\widehat r_j,\widehat\ell_j}^j}$ is not pivotal at $\beta_j=1/3$.
\end{proposition}

\begin{remark}[A tuning-free interval at $\beta_j=1/3$ outside the LSE framework]\label{marginal-remark}
    Although Proposition~\ref{failure-inference} shows that normalization by $\sqrt{n_{\widehat r_j,\widehat\ell_j}^j}$ is not pivotal at $\beta_j=1/3$, a tuning-free interval is still available if one steps outside the isotonic LSE. Define the marginal strip average along the $j$th coordinate by
    \begin{align*}
        \Ybar^{(j)}_{x_{0,j}} = \f{1}{n_{-j}}\sum_{i=1}^n Y_i\,\ind\{x_{i,j}=x_{0,j}\},
        \qquad n_{-j}:=\sum_{i=1}^n \ind\{x_{i,j}=x_{0,j}\}=\prod_{k\neq j}n_k.
    \end{align*}
    By the identifiability constraints and the product structure of the design,
    \[
        \widetilde Y^{(j)}_{x_{0,j}}
        :=\Ybar^{(j)}_{x_{0,j}}-\Ybar
        =f_j^*(x_{0,j})+\vepsbar^{(j)}_{x_{0,j}}-\bar\veps,
        \qquad
        \bar\veps:=\f{1}{n}\sum_{i=1}^n\veps_i.
    \]
    Moreover, $\sqrt{n_{-j}}\bar\veps=O_p(n_j^{-1/2})=o_p(1)$. The central limit theorem and Slutsky's theorem therefore give
    \begin{align*}
        \sqrt{n_{-j}}\big(\widetilde Y^{(j)}_{x_{0,j}}-f_j^*(x_{0,j})\big)\rightsquigarrow N(0,\sigma^2).
    \end{align*}
    Since $n_{-j}$ is observed and, at $\beta_j=1/3$, $\sqrt{n_{-j}}=w_{n,j}^{-1}=n^{1/3}$, the interval
    \[
        \left[\widetilde Y^{(j)}_{x_{0,j}}
        \mathbin{\pm}z_{1-\delta/2}\hat\sigma/\sqrt{n_{-j}}\right]
    \]
    has asymptotic coverage $1-\delta$, requires no tuning parameter, and has the same $n^{-1/3}$ rate as the LSE. This interval uses a single marginal strip and does not exploit pooling induced by monotonicity. It may therefore be less efficient than the LSE, which pools an $O_p(1)$ number of neighboring strips in the critical regime. The marginal interval is rate-optimal only when $\beta_j\le1/3$. When $\beta_j>1/3$, its half-width has order $n^{-(1-\beta_j)/2}\gg n^{-1/3}$. A procedure that is valid and rate-optimal without requiring the asymptotic regime to be specified in advance remains open.
\end{remark}

\section{Numerical examples}\label{simulation-sec}
In this section, we examine our theory in Section \ref{limit-them-section} and Section \ref{inference-section} by simulations. Limiting distributions for $w_{n,1}^{-1}(\fhat_1-f_1^*)(x_{0,1})$ for $\beta_1>1/3$ and $\beta_1=1/3$ in Theorem \ref{theorem-lattice} will be compared to the corresponding empirical distributions. Coverage probabilities of the proposed confidence intervals are verified by simulations. We also compare the length of confidence interval in \eqref{confidence-interval-def} to the length of the oracle confidence interval derived from Theorem \ref{theorem-lattice} assuming the derivative $f_1^{*\prime}(x_{0,1})$ is known. We consider $d=2$ and use a design point nearest to $(0.5,0.5)$ in each lattice. This point equals $(0.5,0.5)$ whenever both coordinate lattice sizes are even. We assume $\sigma$ is known and $\veps_i \stackrel{iid}{\sim} N(0,1)$.
\subsection{Numerical validations for the limiting distribution theory}
We simulate the empirical distributions for the LSEs of the additive monotone model and compare them with the theoretical limiting distributions obtained in Section \ref{limit-them-section}.

We consider the balanced design with $n_1=n_2$ for $\beta_1>1/3$. We first simulate $Y_i=f_1^*(x_{i,1})+f_2^*(x_{i,2})+\veps_i$ for $i=1, \ldots, n_1n_2$ with $f_1^*(x_1) = x_1^2 - n_1^{-1}\sum_{i=1}^{n_1} x_{i,1}^2$ and $f_2^*(x_2) = x_2 - n_2^{-1}\sum_{i=1}^{n_2} x_{i,2}$ (the centering terms enforce the identifiability constraint $\sum_i f_j^*(x_{i,j})=0$), and then calculate the statistic $n^{1/3}(\fhat_1-f_1^*)(x_{0,1})$. Next, we repeat the procedure 1000 times and plot the empirical cumulative distribution function (CDF) of $n^{1/3}(\fhat_1-f_1^*)(x_{0,1})$ in Figure \ref{ecdf-f1} under different lattice sizes. It can be seen from the figure that even for the small lattice size $25\times 25$, the empirical distribution is close to the theoretical limiting distribution.  Since $f_1^{*\prime}(x_{0,1}) = f_2^{*\prime}(x_{0,2})$, Theorem \ref{theorem-lattice} implies that $n^{1/3}(\fhat_1-f_1^*)(x_{0,1})$ and $n^{1/3}(\fhat_2-f_2^*)(x_{0,2})$ have the same limiting distribution. This is empirically confirmed by the QQ-plot of the two statistics in Figure~\ref{qq-plot}, computed with lattice size $50\times 50$.

When $\beta_1=1/3$, we use a Monte Carlo method to approximate the limiting distribution $\mathbb{G}_{1}$ defined in \eqref{gcm-rw-def}.  We sample independent standard normal random variables $\{Z_i\}_{0<|i|\le m}$ and then formulate the truncated two-sided random walk, following the convention in \eqref{gcm-rw-def}:
\begin{align*}
    S_m(r) = \ind_{\{r > 0\}}\sum_{j=1}^{r} Z_j -\ind_{\{r<0\}}\sum_{j=1}^{-r} Z_{-j},
\end{align*}
where $r \in \Z$, $m = 10^3$ and $S_m(0)=0$. Next, we calculate the approximated quantity $\widetilde{\mathbb{G}}_{1}$ according to the following formula:
\begin{align*}
    \widetilde{\mathbb{G}}_{1} = \inf_{r \in [0,m] \cap \Z} \sup_{\ell \in (0, m] \cap \Z} \left\{\f{ S_m(r) - S_m(-\ell)}{r + \ell}+ \f{r-\ell+1}{2}\right\},
\end{align*}
and repeat the process $10^6$ times. In Figure \ref{ecdf-critical}, we compare the empirical CDF of the approximated limiting distribution $\widetilde{\mathbb{G}}_{1}$ and the empirical CDF of $n^{1/3}(\fhat_1 - f_1^*)(x_{0,1})$ with lattice sizes $n_1 = 30, 60, 100$ and $n_2 = 900, 3600, 10000$.
\begin{figure}[htbp]
    \centering
    \begin{subfigure}[t]{0.32\textwidth}
        \centering
        \includegraphics[width=\linewidth, height=3.8cm, keepaspectratio]{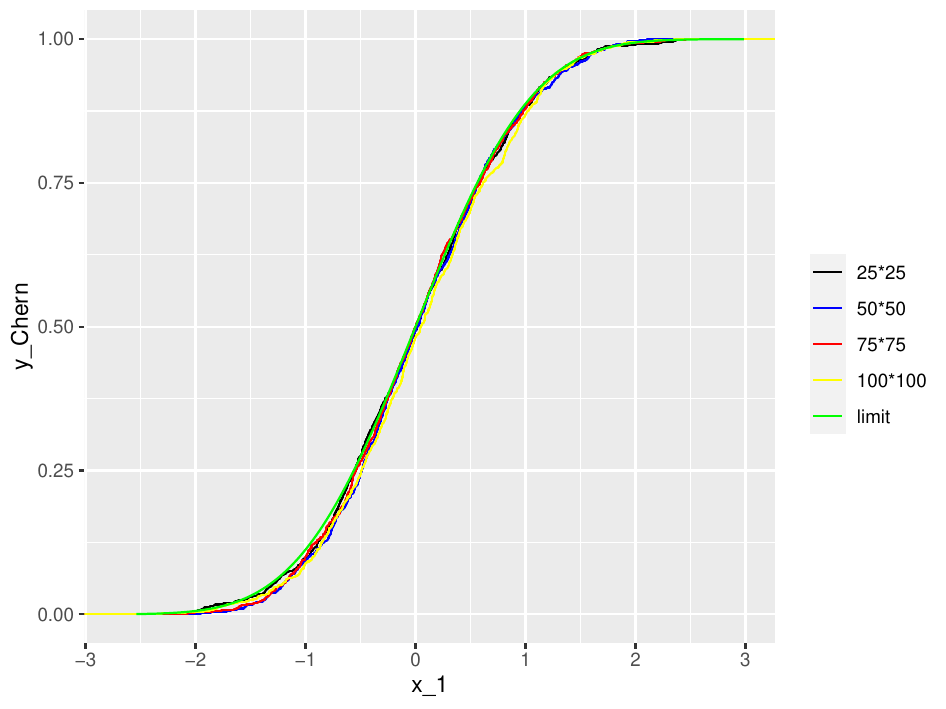}
        \caption{ECDF of $n^{1/3}(\fhat_1-f_1^*)(x_{0,1})$ under different lattice sizes.}
        \label{ecdf-f1}
    \end{subfigure}
    \hfill
    \begin{subfigure}[t]{0.32\textwidth}
        \centering
        \includegraphics[width=\linewidth, height=3.8cm, keepaspectratio]{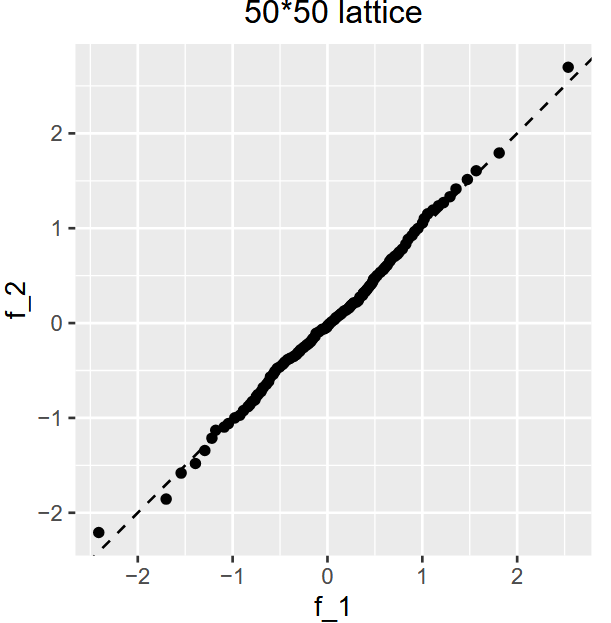}
        \caption{QQ-plot of $n^{1/3}(\fhat_1-f_1^*)(x_{0,1})$ versus $n^{1/3}(\fhat_2-f_2^*)(x_{0,2})$, lattice size $50\times 50$.}
        \label{qq-plot}
    \end{subfigure}
    \hfill
    \begin{subfigure}[t]{0.32\textwidth}
        \centering
        \includegraphics[width=\linewidth, height=3.8cm, keepaspectratio]{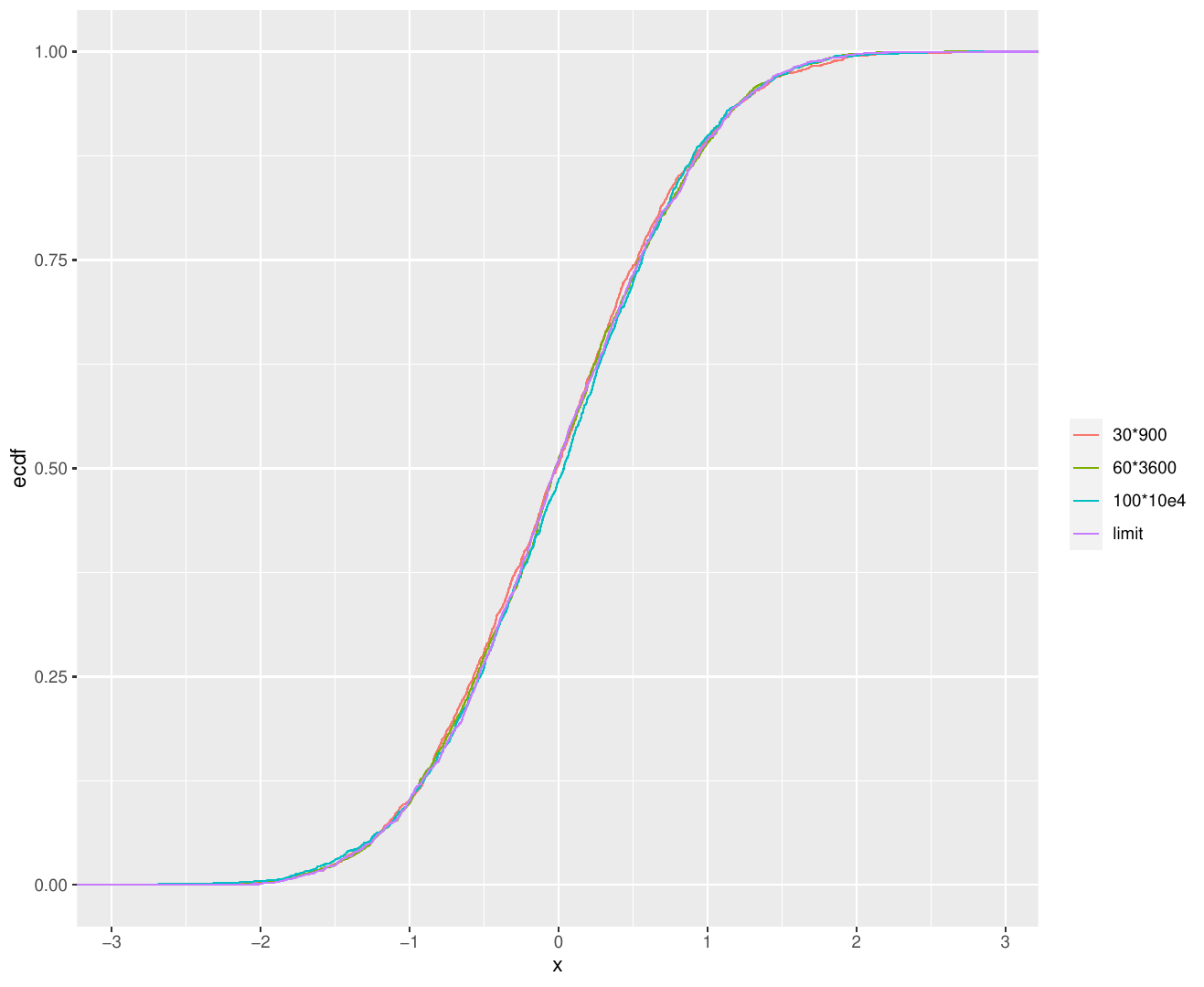}
        \caption{ECDF of $n^{1/3}(\fhat_1-f_1^*)(x_{0,1})$ in the critical regime $\beta_1=1/3$ versus the approximated limit $\widetilde{\mathbb{G}}_1$.}
        \label{ecdf-critical}
    \end{subfigure}
    \caption{Empirical validation of the limiting distributions in Theorem \ref{theorem-lattice}, based on 1000 replications.}
    \label{fig:combined-plots}
\end{figure}

\subsection{Simulations for the confidence interval}
In this subsection, we numerically verify our pivotal limiting distribution theory by plotting the coverage probabilities for each design point and the lengths of confidence intervals for $f_1^*(x_{0,1})$. Let $N$ be the total number of repetitions. For each $b=1,\ldots, N$, construct the confidence interval $\mathcal{I}_n^{(b)}(x;c_{\delta}, 1)$ for each $x\in \calX_1$, where $\mathcal{I}_n^{(b)}(x;c_{\delta}, 1)$ denotes the interval \eqref{confidence-interval-def} for the first coordinate computed from the $b$th replication. The coverage probability at design point $x$ is defined as $N^{-1}\sum_{b=1}^N \ind\{ f_1^*(x) \in  \mathcal{I}_n^{(b)}(x;c_{\delta}, 1)\}$. We focus on 95\% confidence intervals, i.e., $\delta=0.05$, and $c_{\delta}=2.11$ is chosen from Table 4 in \citet{deng2021confidence}.

We simulate $Y_i=f_1^*(x_{i,1})+f_2^*(x_{i,2})+\veps_i$ with the lattice size $100\times 100$, $f_1^*(x_1) = x_1^2 - n_1^{-1}\sum_{i=1}^{n_1} x_{i,1}^2$ and $f_2^*(x_2) = x_2 - n_2^{-1}\sum_{i=1}^{n_2} x_{i,2}$. Next, we calculate the coverage probabilities by repeating the above procedure 1000 times. In Figure \ref{coverage-plot}, we plot the coverage probabilities of $f_1^*(x)$ at all design points. Under the same setting, Figure \ref{box-plot} shows the boxplots for the lengths of the proposed confidence intervals (CIs) under different sample sizes, and the lengths of the oracle CIs are shown as red lines, computed using the true derivative $f_1^{*\prime}(x_{0,1})$. The coverage probabilities are close to $0.95$ for design points near $0.5$, with somewhat larger coverage errors at design points near the boundary of $[0,1]$. The length of the proposed confidence interval is close to the oracle length.

\begin{figure}[htbp]
    \centering
    \begin{subfigure}[t]{0.32\textwidth}
        \includegraphics[width=\linewidth, height=3.8cm, keepaspectratio]{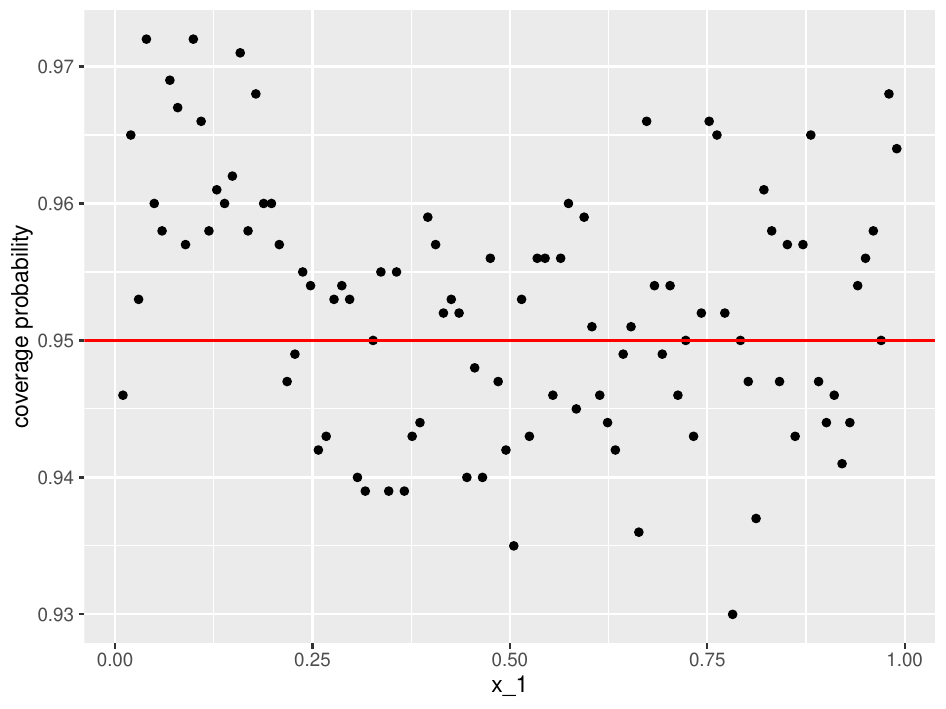}
        \caption{Coverage probabilities of the 95\% CIs for $f_1^*(x)$ at each design point $x \in \calX_1$, lattice size $100\times 100$.}
        \label{coverage-plot}
    \end{subfigure}
    \hfill
    \begin{subfigure}[t]{0.32\textwidth}
        \includegraphics[width=\linewidth, height=3.8cm, keepaspectratio]{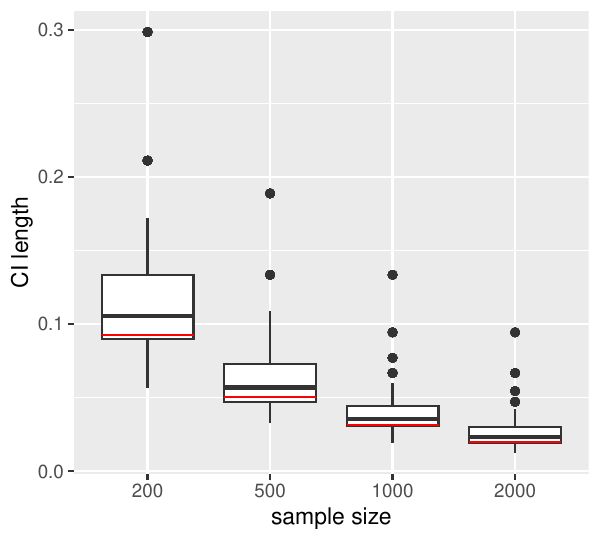}
        \caption{Lengths of the proposed 95\% CIs under different lattice sizes; red lines show the oracle CI lengths based on the true derivative.}
        \label{box-plot}
    \end{subfigure}
    \hfill
    \begin{subfigure}[t]{0.32\textwidth}
        \includegraphics[width=\linewidth, height=3.8cm, keepaspectratio]{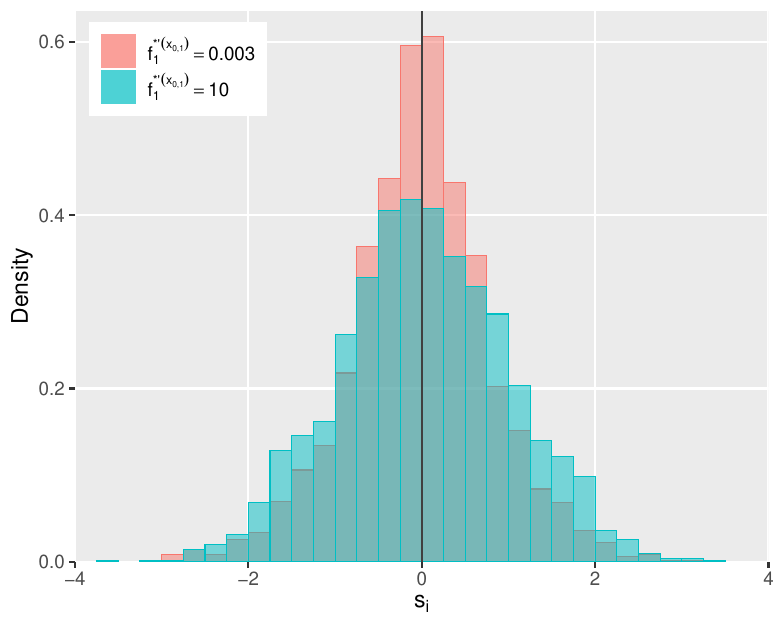}
        \caption{Histograms of the
        $\sqrt{n_{\widehat r_1,\widehat\ell_1}^1}$-normalized statistics
        $s_i$ for derivatives $0.003$ and $10$, shown on a linear density
        scale.}
        \label{histogram-pivotal-compare}
    \end{subfigure}
    \caption{Simulation results for confidence interval coverage and length, together with the comparison of the normalized statistics in the critical regime. Panels (a) and (b) use 1000 replications, and panel (c) uses 2000 replications.}
    \label{fig:combined}
\end{figure}

\subsection{Non-pivotal limiting distribution for \texorpdfstring{$\beta_1=1/3$}{beta1 = 1/3}}\label{simulation-failure}
By Proposition~\ref{failure-inference}, normalization by $\sqrt{n_{\widehat r_1,\widehat\ell_1}^1}$ does not yield a pivotal limit when $\beta_1 = 1/3$. In this subsection, we illustrate numerically that the limiting distribution depends on $f_1^{*\prime}(x_{0,1})$. We simulate from model \eqref{eqn-model} with $d=2$, $n_1 = 200$ and $n_2 = n_1^2 = 40000$, so that $\beta_1 = 1/3$, and with the second component $f_2^*(x_2) = x_2 - n_2^{-1}\sum_{i=1}^{n_2} x_{i,2}$ as in the previous subsections. Let $g_1(x) = 0.003x- n_1^{-1}\sum_{i=1}^{n_1} 0.003x_{i,1}$ and $g_2(x) = 10x - n_1^{-1}\sum_{i=1}^{n_1} 10x_{i,1}$, and let $\fhat_1^{(i)}$ be the LSE corresponding to $f_1^* = g_i$ for $i=1,2$.
We simulate 2000 samples of
$s_i=\sqrt{n_{\widehat r_1,\widehat\ell_1}^1}
\{\fhat_1^{(i)}(x_{0,1})-g_i(x_{0,1})\}$ for $i=1,2$.
Figure~\ref{histogram-pivotal-compare} shows the resulting histograms.
Their different spreads and tail behavior illustrate the dependence on
$f_1^{*\prime}(x_{0,1})$ established in
Proposition~\ref{failure-inference}.

\section{Discussion}\label{discussion-sec}
We establish, to our knowledge, the first limiting distribution theory for the least squares estimators in additive monotone models under general fixed lattice designs, allowing the numbers of design points to differ across coordinates. The main proof device converts the nested min--max characterization into two maxima of independent partial sum processes, yielding the classical isotonic limit, a random-walk limit at the critical resolution, and a Gaussian limit, together with joint asymptotic independence across the additive components. We also construct tuning-free pointwise confidence intervals when $\beta_j>1/3$ and prove that normalization using the square root of the number of observations in the fitted interval is not pivotal at the critical resolution $\beta_j=1/3$. Several directions remain open.

First, our results rely on the Cartesian product structure of the fixed lattice design, which decouples the componentwise estimators via Lemma 3.1 of \citet{guntuboyina2018nonparametric}. Under a random design, this decoupling no longer holds, and existing arguments for the oracle property \citep{mammen2007additive} rely on a uniform boundedness claim for the LSE whose complete proof is, to our knowledge, still unavailable (see Remark~\ref{remark-random-design}). Establishing limiting distributions under random or non-product fixed designs is an interesting open problem.

Second, in the critical regime $\beta_j = 1/3$, Proposition~\ref{failure-inference} shows that normalization by $\sqrt{n_{\widehat r_j,\widehat\ell_j}^j}$ does not produce a pivotal limit for the LSE. The limiting law still depends on the local derivative through the effective lattice spacing. Remark~\ref{marginal-remark} gives a tuning-free interval by averaging a single marginal strip, but this construction does not use pooling induced by monotonicity and may be less efficient than the LSE. It is also rate-optimal only when $\beta_j\le1/3$. Two questions remain open. The first is whether a pivotal interval based on the LSE can be constructed in the critical regime. The second is whether a valid and rate-optimal procedure can be obtained without requiring the asymptotic regime to be specified in advance.

\appendix

\section{Proof of Theorem \ref{theorem-lattice}}\label{sec-proof-lim-them}
Without loss of generality, we set $j=1$ and establish the limiting distribution of the LSE in additive monotone models for $j=1$ across three regimes: (i) $\beta_1 > 1/3$, (ii) $\beta_1 < 1/3$, and (iii) $\beta_1 = 1/3$. We first establish the rate of convergence of the LSE. Next, we derive the asymptotic distribution of a localized version of the estimator, which coincides with the LSE with high probability. Finally, we prove asymptotic joint independence of $w_{n,j}^{-1}\big(\fhat_j(x_{0,j}) - f_j^*(x_{0,j})\big)$ for $j=1,{\ldots},d$.

Since $\fhat_j$ is uniquely defined only at the design points, we extend it by convention to a piecewise constant, left-continuous function. Throughout the proofs, we assume $\mu^* = 0$ without loss of generality: replacing $Y_i$ by $Y_i - \mu^*$ shifts $\muhat$ by $-\mu^*$ and leaves $\fhat_1, \ldots, \fhat_d$ unchanged. We denote the global average of $Y$ by $\Ybar = n^{-1}\sum_{i=1}^n Y_i$. Since $\sum_{i=1}^{n}f_j^*(x_{i,j})=0$ for all $j$, we have $\Ybar= n^{-1}\sum_{i=1}^n\veps_i$. For a design point $x_0 \in \mathcal{X}$, define the local average of $Y$ over the strip $\{x_{i} \in \mathcal{X}: x_{i,j} = x_{0,j}\}$ along the $j$th coordinate as follows:
\begin{align}\label{Ybar-def}
    \Ybar^{(j)}_{x_{0,j}} &= n^{-\sum_{k\neq j}\beta_k}\sum_{i=1}^n Y_{i}\ind\{x_{i,j} = x_{0,j}\}\\
    &= f_j^*(x_{0,j})+\vepsbar_{x_{0,j}}^{(j)},\nonumber
\end{align}
where $\vepsbar_{x_{0,j}}^{(j)} = n^{-\sum_{k\neq j}\beta_k}\sum_{i=1}^n \veps_{i}\ind\{x_{i,j} = x_{0,j}\}$.

Recall from Assumption~\ref{fixed-design-assumption} that
$\calX_j=\{x_{1,j},\ldots,x_{n_j,j}\}\subset[0,1]$ is the set of design
coordinates along the $j$th dimension.
By Lemma 3.1 of \citet{guntuboyina2018nonparametric}, $\fhat_j=\fhat_j^{or}- \Ybar$, where $ \fhat_j^{or}$ is defined as
\begin{align*}
    \fhat_j^{or} \in \argmin_{f_j \in \scrF} \sum_{x \in \calX_j}\left(\Ybar_{x}^{(j)} - f_j(x)\right)^2,
\end{align*}
where the minimization is over the class $\scrF$ of all nondecreasing functions on $[0,1]$ introduced in Section~1. Since $\Ybar = O_p(n^{-1/2})$ and the local rate of convergence of $\fhat_j$ is of larger order than $n^{-1/2}$ (see \eqref{rate-cvgc-def}), it suffices to study the limiting distribution of $\fhat_j^{or}$ to obtain the limiting distribution of $\fhat_j$. In the sequel, we establish the limiting distribution for $w_{n,1}^{-1}(\fhat_1(x_{0,1}) - f_1^*(x_{0,1}))$, and the same proof applies to $w_{n,j}^{-1}(\fhat_j(x_{0,j}) - f_j^*(x_{0,j}))$ for $j \neq 1$. For ease of notation, we drop superscripts by writing $\Ybar_{x_{i,1}}$ as shorthand for $\Ybar_{x_{i,1}}^{(1)}$ in \eqref{Ybar-def} and $\vepsbar_{x_{i,1}}$ as shorthand for $\vepsbar_{x_{i,1}}^{(1)}$ for $x_{i,1} \in \calX_1$.

Define the continuous piecewise linear function $F_{n,1}$ as follows.
\begin{align*}
    F_{n,1}(0):=0 \quad \text{ and } \quad F_{n,1}\big(i/n_1\big) := \f{1}{n_1} \sum_{j=1}^{i}\Ybar_{x_{j,1}}, \quad \text{ for } i=1,{\ldots}, n_1,
\end{align*}
and $F_{n,1}(t)$ as the linear interpolation of $F_{n,1}(i/n_1)$ between the values $F_{n,1}(i/n_1)$ for $i = 1,{\ldots}, n_1$. The least squares estimator can be written as the left-hand slope of the greatest convex minorant of $F_{n,1}$ and hence admits the following min--max representation (see Section~3.3 of \citet{groeneboom2014nonparametric}). For any $w_{n,1}>0$, we have
 \begin{align}
    \fhat_1^{or}(x_{0,1}) &= \min_{r \ge 0}\max_{\ell > 0}\f{F_{n,1}(x_{0,1}+rw_{n,1}) - F_{n,1}(x_{0,1} - \ell w_{n,1})}{(r+\ell)w_{n,1}},\label{f-oracle}\\
    &=\f{F_{n,1}(x_{0,1}+\rhat_1 w_{n,1}) - F_{n,1}(x_{0,1} - \ellhat_1 w_{n,1})}{(\rhat_1+\ellhat_1)w_{n,1}}\nonumber\\
    &= \min_{a \ge x_{0,1}} \max_{b \le x_{0,1}} \f{\sum_{i=1}^{n_1}\Ybar_{x_{i,1}}\ind\{x_{i,1} \in [b, a] \cap \calX_1\}}{\sum_{i=1}^{n_1}\ind\{x_{i,1} \in [b, a]\cap \calX_1\}},\label{min-max-formula}
\end{align}
The half-open interval
$(x_{0,1}-\ellhat_1w_{n,1},x_{0,1}+\rhat_1w_{n,1}]$
contains exactly the consecutive design points in the fitted constant interval.
Thus, $\rhat_1$ and $\ellhat_1$ are its scaled right and left widths. By the
piecewise linearity of $F_{n,1}$, the optimizing primitive endpoints may be
chosen from the knot set $\{0,1/n_1,\ldots,1\}$. Hence,
$x_{0,1}+\rhat_1w_{n,1}$ and $x_{0,1}-\ellhat_1w_{n,1}$ are grid knots.
The left knot may equal zero and need not belong to $\calX_1$.

The scaled right and left widths $(\rhat_1,\ellhat_1)$ of the fitted constant interval form a saddle pair for the objective in \eqref{f-oracle}. For fixed $\rhat_1$, the objective is maximized at $\ellhat_1$, and for fixed $\ellhat_1$, it is minimized at $\rhat_1$. This follows from the standard equality of the min--max and max--min characterizations of the isotonic LSE (see Section~3.3 of \citet{groeneboom2014nonparametric}).

Let $c>0$ and define the one-sided localized estimator $\fhat_{1, c}^{or}$ as follows.
\begin{align}\label{f-oracle-local}
    \fhat_{1, c}^{or}(x_{0,1}) =
    \begin{cases}    
    \min\limits_{0 \le r \le c}\max\limits_{0 < \ell \le c}\f{F_{n,1}(x_{0,1}+rw_{n,1}) - F_{n,1}(x_{0,1} - \ell w_{n,1})}{(r+\ell)w_{n,1}}, & \beta_1 \ge 1/3\\
    \Ybar_{x_{0,1}}, & \beta_1 <1/3.
    \end{cases}
\end{align}

\subsection{Local rate of convergence}
In this subsection, we establish the rate of convergence of $\fhat_{1}^{or}(x_{0,1})$. For an interval $I\subset \R$ with $I\cap\calX_1\neq \emptyset$, define
\begin{align*}
    \overline{f_1^*}\big|_{I}
    &:= \f{\sum_{i=1}^n f_1^*(x_{i,1})\ind\{x_{i,1}\in I\}}{\sum_{i=1}^n\ind\{x_{i,1}\in I\}},\\
    \vepsbar\big|_{I}
    &:= \f{\sum_{i=1}^n \veps_i\,\ind\{x_{i,1}\in I\}}{\sum_{i=1}^n \ind\{x_{i,1}\in I\}}.
\end{align*}

\begin{proposition}\label{rate-cvgc}
    Assume the same conditions as in Theorem \ref{theorem-lattice}. With $w_{n,1}$ defined in \eqref{rate-cvgc-def}, we have $\fhat_{1}^{or}(x_{0,1}) - f_1^*(x_{0,1}) = O_p(w_{n,1})$.
\end{proposition}
\begin{proof}

Write $A(r,\ell)$ for the objective in \eqref{f-oracle},
and set $d_n=(n_1w_{n,1})^{-1}$. First suppose that
$\beta_1\ge1/3$, so that $n_1^{-1}\le w_{n,1}$. Set
\[
    r_0=\ell_0:=\lceil n_1w_{n,1}\rceil/(n_1w_{n,1})\in[1,2),
\]
so that $x_{0,1}+r_0w_{n,1}$ and $x_{0,1}-\ell_0w_{n,1}$ are knots of
$F_{n,1}$. By the saddle-point property of
$(\rhat_1,\ellhat_1)$,
\[
    \fhat_1^{or}(x_{0,1})-f_1^*(x_{0,1})
    \le A(r_0,\ellhat_1)-f_1^*(x_{0,1}).
\]
Since the endpoints of the fitted interval are knots, we have
$\ellhat_1\in d_n\Z_{\ge1}$, so both endpoints of the window
$I=(x_{0,1}-\ellhat_1w_{n,1},\,x_{0,1}+r_0w_{n,1}]$ are knots and
$A(r_0,\ellhat_1)=\overline{f_1^*}\big|_{I}+\vepsbar\big|_{I}$. By
monotonicity of $f_1^*$ and Assumption~\ref{smoothness-assumption},
\[
    \overline{f_1^*}\big|_{I}
    \le f_1^*(x_{0,1}+r_0w_{n,1})
    \le f_1^*(x_{0,1})+2C_1w_{n,1},
\]
   where the second inequality follows from Assumption~\ref{smoothness-assumption} for some constant $C_1>0$. Since
$r_0\ge1$, Lemma~\ref{error-control} bounds the noise average uniformly
over $\ellhat_1$:
\[
    \big|\vepsbar\big|_{I}\big|
    \le\sup_{\ell\ge d_n}
    \big|\vepsbar\big|_{(x_{0,1}-\ell w_{n,1},\,x_{0,1}+r_0w_{n,1}]}\big|
    =O_p(w_{n,1}).
\]
Consequently,
\[
    \fhat_1^{or}(x_{0,1})-f_1^*(x_{0,1})
    \le 2C_1w_{n,1}+O_p(w_{n,1}).
\]
Taking $\ell=\ell_0$ in the other saddle-point
inequality, bounding the deterministic average of the
window $(x_{0,1}-\ell_0w_{n,1},\,x_{0,1}+\rhat_1w_{n,1}]$, whose endpoints are knots,
from below by $f_1^*(x_{0,1}-\ell_0w_{n,1})\ge f_1^*(x_{0,1})-2C_1w_{n,1}$,
and applying Lemma~\ref{error-control} with the two endpoints
interchanged gives
\[
    \fhat_1^{or}(x_{0,1})-f_1^*(x_{0,1})
    \ge -2C_1w_{n,1}-O_p(w_{n,1}).
\]

Now suppose that $\beta_1<1/3$. In this regime one coordinate strip contains
$n_{-1}=w_{n,1}^{-2}$ observations. Taking $r=0$ in the upper saddle-point
inequality and $\ell=d_n$ in the lower saddle-point inequality gives
\begin{align*}
    \fhat_1^{or}(x_{0,1})-f_1^*(x_{0,1})
    &\le
      \sup_{\ell\ge d_n}
      \left|\vepsbar\big|_{(x_{0,1}-\ell w_{n,1},x_{0,1}]}\right|,\\
    \fhat_1^{or}(x_{0,1})-f_1^*(x_{0,1})
    &\ge
      -\sup_{r\ge0}
      \left|\vepsbar\big|_{(x_{0,1}-d_nw_{n,1},
                                 x_{0,1}+rw_{n,1}]}\right|.
\end{align*}
Monotonicity removes the deterministic terms in the indicated directions,
and Lemma~\ref{error-control} makes both suprema $O_p(w_{n,1})$. This
completes the proof.

\end{proof}

\subsection{Localizing the estimator}
We want to show that $\fhat_1^{or}$ is equal to the localized version $\fhat_{1,c}^{or}(x_{0,1})$ with high probability, where $\fhat_{1,c}^{or}(x_{0,1})$ is defined in \eqref{f-oracle-local}.
\begin{proposition}\label{large-localization-prop}
    Under the same conditions as Theorem \ref{theorem-lattice}, 
    {
    \begin{equation*}
        \lim_{c \rightarrow \infty}\limsup_{n \rightarrow  \infty}
        \P\Big\{\fhat_1^{or}(x_{0,1})\ne\fhat_{1,c}^{or}(x_{0,1})\Big\}=0,
    \end{equation*}
    }
    where $\fhat_1^{or}(x_{0,1})$ and $\fhat_{1,c}^{or}(x_{0,1})$ are defined in \eqref{f-oracle} and \eqref{f-oracle-local} respectively.
\end{proposition}

\begin{proof}

Write $A(r,\ell)=A_f(r,\ell)+A_{\varepsilon}(r,\ell)$ for the
deterministic and stochastic parts of the objective in \eqref{f-oracle}.
Set $d_n=(n_1w_{n,1})^{-1}$ and
$\ell_0=\lceil n_1w_{n,1}\rceil/(n_1w_{n,1})$, so that
$x_{0,1}-\ell_0w_{n,1}$ is a knot of $F_{n,1}$. Then $\ell_0\in[1,2)$
when $\beta_1\ge1/3$, and $\ell_0=d_n$ when $\beta_1<1/3$.
First suppose that $\beta_1\ge1/3$. It suffices to show that
\[
    \lim_{c\to\infty}\limsup_{n\to\infty}
    \P\{\rhat_1>c\text{ or }\ellhat_1>c\}=0.
\]
Suppose that $\rhat_1>c>1$. By monotonicity of $f_1^*$, the deterministic part of $F_{n,1}$ is piecewise linear and convex, so the deterministic
objective $A_f(r,\ell_0)$ is nondecreasing in $r$. The lower saddle-point
inequality with $\ell=\ell_0$ therefore gives
\begin{align}
    \fhat_1^{or}(x_{0,1})-f_1^*(x_{0,1})
    &\ge A(\rhat_1,\ell_0)-f_1^*(x_{0,1})\nonumber\\
    &\ge A_f(c,\ell_0)-f_1^*(x_{0,1})
       +A_{\varepsilon}(\rhat_1,\ell_0).
       \label{large-dev-low-bound1}
\end{align}
Since $\beta_1\ge1/3$, the lattice spacing satisfies
$n_1^{-1}\le w_{n,1}$. For each fixed $c$ and sufficiently large $n$, we
may therefore replace the right endpoint $x_{0,1}+cw_{n,1}$ by an adjacent knot of $F_{n,1}$, while the left endpoint $x_{0,1}-\ell_0w_{n,1}$ is already a knot.
By Assumption~\ref{smoothness-assumption}, this changes the deterministic
secant slope by at most $C_3w_{n,1}$. Hence, for sufficiently large $c$
and then all sufficiently large $n$,
\begin{equation}
    A_f(c,\ell_0)-f_1^*(x_{0,1})
    \ge C_2cw_{n,1}-C_3w_{n,1}
    \ge C_4cw_{n,1},                         \label{large-dev-low-bound2}
\end{equation}
where $C_2,C_3,C_4>0$ do not depend on $c$ or $n$.

Define
\begin{equation}
    \Xi_n
    :=w_{n,1}^{-1}\sup_{r\ge0}|A_{\varepsilon}(r,\ell_0)|.
    \label{eqn-25}
\end{equation}
As $r$ moves between two consecutive knots, the numerator and
the denominator of $A_{\varepsilon}(r,\ell_0)$ are both affine in $r$, so
$A_{\varepsilon}(r,\ell_0)$ is monotone on each such segment. The supremum
in \eqref{eqn-25} therefore equals the supremum over the knot widths
$r\in d_n\Z_{\ge0}$, and for such $r$ both endpoints of the averaging
window are knots.
Lemma~\ref{error-control} with the two endpoints interchanged then gives
$\Xi_n=O_p(1)$ in both regimes, and $\Xi_n$ does not depend on $c$.
Hence, on $\{\rhat_1>c\}$,
\[
    w_{n,1}^{-1}
    \{\fhat_1^{or}(x_{0,1})-f_1^*(x_{0,1})\}
    \ge C_4c-\Xi_n.
\]
Consequently,
\begin{align*}
    \P\{\rhat_1>c\}
    &\le
    \P\left\{
      w_{n,1}^{-1}
      \{\fhat_1^{or}(x_{0,1})-f_1^*(x_{0,1})\}
      \ge \f{C_4c}{2}
    \right\}\\
    &\quad+
    \P\left\{\Xi_n\ge\f{C_4c}{2}\right\}.
\end{align*}
Both random variables on the right are $O_p(1)$, the first by
Proposition~\ref{rate-cvgc}. Therefore,
\[
    \lim_{c\to\infty}\limsup_{n\to\infty}\P\{\rhat_1>c\}=0.
\]
The same argument with the endpoints interchanged proves the corresponding
claim for $\ellhat_1$.

Now suppose that $\beta_1<1/3$. The possible
endpoint widths are integer multiples of $d_n$, and $d_n\to\infty$.
On $\{\rhat_1\ge d_n\}$, the window whose endpoints are knots
$(x_{0,1}-n_1^{-1},\,x_{0,1}+\rhat_1w_{n,1}]$ contains the strip at
$x_{0,1}$ and at least one strip above it. Since
$f_1^{*\prime}(x_{0,1})>0$, the lower saddle-point inequality with
$\ell=\ell_0=d_n$ gives
\begin{align}
    \fhat_1^{or}(x_{0,1})-f_1^*(x_{0,1})
    \ge Cn_1^{-1}+A_{\varepsilon}(\rhat_1,\ell_0).
    \label{large-dev-low-bound3}
\end{align}
It follows from \eqref{eqn-25} that
\begin{align*}
    \P\{\rhat_1\ge d_n\}
    &\le
    \P\left\{
      w_{n,1}^{-1}
      \{\fhat_1^{or}(x_{0,1})-f_1^*(x_{0,1})\}
      \ge \f C2\f{n_1^{-1}}{w_{n,1}}
    \right\}\\
    &\quad+
    \P\left\{
      \Xi_n\ge \f C2\f{n_1^{-1}}{w_{n,1}}
    \right\}.
\end{align*}
Since both random variables are $O_p(1)$ and
$n_1^{-1}/w_{n,1}\to\infty$, the two probabilities tend to zero. Thus,
$\P\{\rhat_1=0\}\to1$. Interchanging the endpoints and starting from
$\{\ellhat_1\ge2d_n\}$ similarly gives
$\P\{\ellhat_1=d_n\}\to1$. In this regime the full estimator therefore
equals $\Ybar_{x_{0,1}}$ with probability tending to one, which is the
localized definition in \eqref{f-oracle-local}. This completes the proof.

\end{proof}
	        
\subsection{Limiting distribution for \texorpdfstring{$\beta_1>1/3$}{beta1 > 1/3}}\label{lim-Chernoff-sec}
We now prove the limiting distribution of $\fhat_1^{or}$ for $\beta_1>1/3$.
\begin{proof}[Proof of Theorem \ref{theorem-lattice}, case $\beta_1>1/3$]
By \eqref{f-oracle-local} and Lemma \ref{lm-switch}, we have
\begin{align}
    &\P\biggl\{ w_{n,1}^{-1}\Big[  \fhat_{1,c}^{or}(x_{0,1})-f_1^*(x_{0,1}) \Big] \le t \biggr\} \nonumber\\
    &= \P\left\{ w_{n,1}^{-1}\left[ \min_{r \in [0,c]} \max_{\ell \in (0,c]}  \f{F_{n,1}(x_{0,1}+rw_{n,1}) - F_{n,1}(x_{0,1} - \ell w_{n,1})}{(r+\ell)w_{n,1}} - f_1^*(x_{0,1}) \right]\le t\right\}\nonumber\\
    &=\P\left\{ \min_{r \in [0,c]} \max_{\ell \in (0,c]}  w_{n,1}^{-2} \left(F_{n,1}(x_{0,1}+rw_{n,1})-F_{n,1}(x_{0,1}) +F_{n,1}(x_{0,1})\right. \right. \nonumber\\
    &\left. \left.\quad \quad \quad\quad\quad\quad\quad\quad\quad - F_{n,1}(x_{0,1} - \ell w_{n,1}) - (r+\ell)w_{n,1} f_1^*(x_{0,1})\right)- (r+\ell)t  \le 0 \right\}\nonumber\\
    &:=\P\{-R_n + L_n \le 0\},\label{eqn-26}
\end{align}
where 
\begin{align}
    R_n &= \max_{r \in [0,c]} w_{n,1}^{-2} \bigl(-F_{n,1}(x_{0,1}+rw_{n,1}) +F_{n,1}(x_{0,1})+rw_{n,1}f_1^*(x_{0,1})+w_{n,1}^2rt \bigr),\label{Rn-def}\\
    L_n &= \sup_{\ell \in (0,c]}w_{n,1}^{-2} \bigl( - F_{n,1}(x_{0,1}-\ell w_{n,1})+F_{n,1}(x_{0,1}) - \ell w_{n,1}f_1^*(x_{0,1})-w_{n,1}^{2} \ell t \bigr)\label{Ln-def}.
\end{align}

Write $a=f_1^{*\prime}(x_{0,1})$ and define the centered and rescaled
two-sided cumulative sum process, for every $s$ such that
$x_{0,1}+sw_{n,1}\in[0,1]$, by
\begin{align}
    Q_n(s)
    &:=
    w_{n,1}^{-2}
    \left\{
      F_{n,1}(x_{0,1}+sw_{n,1})-F_{n,1}(x_{0,1})
      -sw_{n,1}f_1^*(x_{0,1})
    \right\}. \label{two-sided-Qn}
\end{align}
For every fixed $c>0$, the interval $[-c,c]$ is contained in this domain
for all sufficiently large $n$. The linear interpolation in $F_{n,1}$
already accounts for a fractional terminal strip. A uniform Taylor
expansion under Assumption~\ref{smoothness-assumption} gives
\[
    \sup_{|s|\le c}
    \left|\E Q_n(s)-\f a2s^2\right|\longrightarrow0.
\]
After translating the strip index to $x_{0,1}$, Lemma~\ref{weak-convergence-lemma} applies separately to the disjoint right and left strips. Its proof applies on every fixed $[0,c]$ after deterministic time rescaling. The two halves are independent for every $n$, and hence
\begin{align}
    Q_n\rightsquigarrow Q
    \quad\text{in }C([-c,c]),
    \qquad
    Q(s):=\sigma W(s)+\f a2s^2, \label{two-sided-Q-conv}
\end{align}
where $W$ is a two-sided Brownian motion with independent halves.

The definitions \eqref{Rn-def}--\eqref{Ln-def} give the exact identities
\[
    R_n=\sup_{0\le r\le c}\{-Q_n(r)+tr\},
    \qquad
    L_n=\sup_{0<\ell\le c}\{-Q_n(-\ell)-t\ell\}.
\]
Since $Q_n$ is continuous at zero and $Q_n(0)=0$,
$-Q_n(-\ell)-t\ell\to0$ as $\ell\downarrow0$. Therefore, adjoining
$\ell=0$ does not change the second supremum.
For $q\in C([-c,c])$, define
\[
    \Phi(q):=
    \left(
      \sup_{0\le r\le c}\{-q(r)+tr\},
      \sup_{0\le\ell\le c}\{-q(-\ell)-t\ell\}
    \right).
\]
The bound
\[
    \|\Phi(q)-\Phi(\widetilde q)\|_\infty
    \le\|q-\widetilde q\|_\infty
\]
shows that $\Phi$ is continuous. Therefore, \eqref{two-sided-Q-conv} and the continuous mapping theorem give
\begin{align}
    (R_n,L_n)
    \rightsquigarrow
    \left(
      \sup_{0\le r\le c}\left\{-\sigma W(r)-\f a2r^2+tr\right\},
      \sup_{0\le\ell\le c}\left\{-\sigma W(-\ell)-\f a2\ell^2-t\ell\right\}
    \right). \label{joint-Rn-Ln}
\end{align}

Recall $Q(s)=\sigma W(s)+\f a2 s^2$ from \eqref{two-sided-Q-conv}. The joint convergence \eqref{joint-Rn-Ln}, \eqref{eqn-26}, and Lemma~\ref{lm-switch} yield, as $n\to\infty$,
\begin{align*}
    &\P\left\{ w_{n,1}^{-1}\big(\fhat^{or}_{1,c}(x_{0,1})-f_1^*(x_{0,1})\big) \le t \right\} \\
    &\rightarrow \P\left\{ \min_{r \in [0,c]}\max_{\ell \in (0,c]} \sigma \left(W(r)-W(-\ell)\right)+\f{f_1^{*\prime}(x_{0,1})( r^2 - \ell^2)}{2} - t(r+\ell) \le 0\right\}\\
    &=\P\left\{\min_{r \in [0,c]}\max_{\ell \in (0,c]}   \f{Q(r) - Q(-\ell)}{r+\ell}\le t  \right\},
\end{align*}
where the last equality follows by Lemma \ref{lm-switch}.
The two suprema in \eqref{joint-Rn-Ln} are independent, since $W$ has
independent left and right halves. Each has a continuous distribution as the
supremum of Brownian motion plus a continuously differentiable drift on a
nondegenerate compact interval \citep{morters2010brownian}. Their difference
therefore has a continuous distribution, so $0$ is a continuity point of the
limit of $L_n - R_n$ in \eqref{eqn-26}.
By Proposition \ref{large-localization-prop}, we have
\begin{align*}
    \lim_{c \rightarrow \infty}\limsup_{n \rightarrow \infty}
    \P\bigl\{\fhat_1^{or}(x_{0,1}) \neq \fhat_{1,c}^{or}(x_{0,1})\bigr\}
    = 0.
\end{align*}
Define 
\begin{align*}
    V_c &= \min_{r \in [0,c]}\max_{\ell \in (0,c]} \f{Q(r) - Q(-\ell)}{r+\ell},\\
    V &= \min_{r \ge 0} \max_{\ell > 0} \f{Q(r) - Q(-\ell)}{r+\ell}.
\end{align*}
By Proposition 7 in \citet{han2020limit}, we know that
\begin{align*}
    \lim_{c\rightarrow \infty}\P\{ V_c  \neq V \} =0. 
\end{align*}
Therefore, the Lemma~\ref{delocalization-lemma} from \citet{rao1969estimation} gives $ w_{n,1}^{-1}(\fhat^{or}_{1}(x_{0,1}) -f_1^*(x_{0,1})) \rightsquigarrow V$. Note that $V$  \citep{rao1969estimation}  can be written as 
\begin{align*}
    V &= gcm\left(\sigma W(h)+ \f{f_1^{*\prime}(x_{0,1})h^2}{2} \right)'_-(0)\\
      & \stackrel{d}{=} \Big[ 4\sigma^2f^{*\prime}_1(x_{0,1}) \Big]^{1/3}\mathbb{C},
\end{align*}
where $\mathbb{C}=\argmin_{h \in \R}\{W(h)+h^2\}$ is Chernoff's distribution and $W$ is a two-sided Brownian motion starting from zero. 
\end{proof}

\subsection{Limiting distribution for \texorpdfstring{$\beta_1<1/3$}{beta1 < 1/3}}\label{data-poor-regime}

\begin{proof}[Proof of Theorem \ref{theorem-lattice}, case $\beta_1<1/3$]
By \eqref{rate-cvgc-def}, the convergence rate satisfies $w_{n,1} = n^{-(1-\beta_1)/2}$, which is of smaller order than $n_1^{-1}$. On the event $\{\rhat_1 = 0,\ \ellhat_1 \le (n_1w_{n,1})^{-1} \}$, the increment of $F_{n,1}$ in \eqref{f-oracle} spans only the strip at $x_{0,1}$, so $\fhat_1^{or}(x_{0,1})=\Ybar_{x_{0,1}}=f_1^*(x_{0,1})+\vepsbar_{x_{0,1}}$ and 
\begin{align*}
	&w_{n,1}^{-1}(\fhat_1^{or}(x_{0,1}) - f_1^*(x_{0,1})) \\
		&=w_{n,1}^{-1}\vepsbar_{x_{0,1}} = n^{(1-\beta_1)/2}\vepsbar_{x_{0,1}}\\
	&\rightsquigarrow N(0,\sigma^2).
\end{align*}
By Proposition \ref{large-localization-prop}, we know that $\lim_{n \rightarrow \infty}\P\{ \rhat_1 = 0,\ \ellhat_1 \le (n_1w_{n,1})^{-1} \} = 1$. Therefore, $w_{n,1}^{-1}(\fhat_1(x_{0,1}) - f_1^*(x_{0,1})) \rightsquigarrow N(0, \sigma^2)$.
\end{proof}

\subsection{Limiting distribution for \texorpdfstring{$\beta_1 = 1/3$}{beta1 = 1/3}}
We first prove that the limiting distribution under $\beta_1 = 1/3$ can be localized.
\begin{proposition}\label{local-limit-critical}
    For $t \in \R$, $c>0$ and $f_1^{*\prime}(x_{0,1})>0$, define 
    \begin{align*}
        G_c &=\min_{r \in [0,c] \cap \Z} \max_{\ell \in [1, c] \cap \Z} \sigma S(r) - \sigma S(-\ell) + \f{f_1^{*\prime}(x_{0,1})(r(r+1) - \ell(\ell-1))}{2} - t(r+\ell),\\
        G &= \min_{r \in \Z_{\ge 0}} \max_{\ell \in  \Z_{\ge 1}} \sigma S(r) - \sigma S(-\ell) + \f{f_1^{*\prime}(x_{0,1})(r(r+1) - \ell(\ell-1))}{2} - t(r+\ell),
    \end{align*}
    where $S(\cdot)$ is a two-sided random walk with standard Gaussian increments. Then $\P\{ G_c \neq G \} \rightarrow 0$ as $c \rightarrow \infty$.
\end{proposition}
\begin{proof}
    We define
    \begin{align*}
        g_c^+ &= \min_{r \in [0,c] \cap \Z} \sigma S(r)  +\f{f_1^{*\prime}(x_{0,1})r(r+1)}{2} - tr,\\
        g_c^- &= \max_{\ell \in [1,c] \cap \Z} -\sigma S(-\ell) - \f{f_1^{*\prime}(x_{0,1})\ell(\ell-1)}{2} - t\ell,\\
        g^+ &= \inf_{r \in \Z_{\ge 0}} \sigma S(r)  +\f{f_1^{*\prime}(x_{0,1})r(r+1)}{2} - tr,\\
        g^- &= \sup_{\ell \in \Z_{\ge 1} } -\sigma S(-\ell) -  \f{f_1^{*\prime}(x_{0,1})\ell(\ell-1)}{2} - t\ell.
    \end{align*}
    Then $G_c = g_c^+ + g_c^-$ and $G = g^+ + g^-$. It suffices to show that $\lim_{c \rightarrow \infty}\P\{ g_c^+ \neq g^+ \}=0$ and $\lim_{c \rightarrow \infty}\P\{ g_c^- \neq g^- \}=0$. We consider $g_c^+$. The argument for $g_c^-$ is analogous.

    Define $G^+(r)$ as the linear interpolation of
    $\sigma S(r)+(f_1^{*\prime}(x_{0,1})r(r+1))/2-tr$ over integers.
    We first prove that $\inf_{r\ge0}G^+(r)$ is attained and that its
    minimizer is unique almost surely. There exist constants $\eta>0$ and $r_0 \ge 1$, depending only on $f_1^{*\prime}(x_{0,1})$ and $t$, such that $(f_1^{*\prime}(x_{0,1})r(r+1))/2 - tr \ge 2\eta r^2$ for all $r \ge r_0$. By the law of the iterated logarithm, $|S(r)|/r^2 \to 0$ almost surely as $r \to \infty$, so
    \begin{align*}
        G^+(r) \ge \eta r^2 \quad \text{for all } r \ge R_0
    \end{align*}
    for some almost surely finite random variable $R_0 \ge r_0$. In particular, $G^+(r) \rightarrow \infty$ as $r \rightarrow \infty$ almost surely, and the infimum of $G^+$ is attained.
    Since $G^+$ is a Gaussian process with continuous sample paths and $\Var(G^+(y) -G^+(z) )\neq 0$ for $y \neq z$,
    by Lemma 2.6 in \citet{kim1990cube}, there exists a unique $r^*$ such that $g^+= G^+(r^*)$ almost surely.

    Define, for any $c>0$,
    \begin{align*}
        r_c^* = \argmin_{r \in [0,c]}G^+(r),
        \qquad
        r^* = \argmin_{r\ge 0}G^+(r).
    \end{align*}
    In order to prove $\lim_{c \rightarrow \infty}\P\{ g_c^+ \neq g^+ \}=0$, it suffices to prove $\lim_{c \rightarrow \infty}\P\{ r^* \neq r_c^* \} = 0$. Notice that $\min_{r \ge 0} G^+(r) \le G^+(1)$, and on the event $\{R_0 \le c\}$ we have $\inf_{r>c}G^+(r) \ge \eta c^2$. Thus,
    \begin{align*}
        \P\{ r^* \neq r_c^* \} &\le \P\{ G^+(1) \ge \inf_{r >c} G^+(r)\}\\
        &\le \P\{R_0 > c\} + \P\{\sigma Z_1 + f_1^{*\prime}(x_{0,1}) - t \ge \eta c^2\},
    \end{align*}
    which goes to $0$ as $c \rightarrow \infty$ since $R_0$ is almost surely finite.
    Therefore, $\lim_{c \rightarrow \infty} \P\{ g_c^+ \neq g^+ \} = 0$ by the fact that the minimizer of $G^+(r)$ is attained on integers by the linear interpolation.

\end{proof}

Now, we present our proof of the limiting distribution of $\fhat_1^{or}$ for $\beta_1=1/3$.
\begin{proof}[Proof of Theorem \ref{theorem-lattice}, case $\beta_1=1/3$]
Since $w_{n,1}^{-1}=n_1$, the min-max operator can be defined over the integers. We will adopt the same definitions of $R_n$ and $L_n$ in \eqref{Rn-def} and \eqref{Ln-def}. For an integer $c$, as in the proof for $\beta_1>1/3$, we have

\begin{align*}
    R_n & = \max_{r \in [0,c] \cap \Z} -\sum_{x \in (x_{0,1}, x_{0,1}+rw_{n,1}] \cap \calX_1}w_{n,1}^{-1}\left[ f_1^{*\prime}(x_{0,1})(x - x_{0,1})(1+o(1)) +\vepsbar_x\right]+rt+o_p(1)\\
    &=\max_{r \in [0,c] \cap \Z} \f{-f_1^{*\prime}(x_{0,1})r(1+r)}{2}(1+o(1))- w_{n,1}^{-1}\sum_{x \in (x_{0,1}, x_{0,1}+rw_{n,1}] \cap \calX_1}\vepsbar_x + rt+o_p(1),
\end{align*}
and 
\begin{align*}
    L_n & = \max_{\ell \in (0,c] \cap \Z} \sum_{x \in (x_{0,1}-\ell w_{n,1}, x_{0,1}] \cap \calX_1}w_{n,1}^{-1}\left[ f_1^{*\prime}(x_{0,1})(x - x_{0,1})(1+o(1)) +\vepsbar_x\right]-\ell t+o_p(1)\\
    &= \max_{\ell \in (0,c] \cap \Z} \f{-f_1^{*\prime}(x_{0,1})\ell(\ell-1)}{2}(1+o(1))+w_{n,1}^{-1}\sum_{x \in (x_{0,1}-\ell w_{n,1},x_{0,1}] \cap \calX_1}\vepsbar_x-\ell t+o_p(1).
\end{align*}
Since the $\vepsbar_x$ are independent, by the multivariate central limit theorem, for $x_{1,1}, {\ldots}, x_{k,1} \in \calX_1$,
\begin{align*}
    w_{n,1}^{-1}\left(\vepsbar_{x_{1,1}}, \ldots, \vepsbar_{x_{k,1}} \right) \rightsquigarrow N(0, \sigma^2 I_k),
\end{align*}
where $I_k \in \R^{k \times k}$ denotes the $k\times k$ identity matrix. Therefore,
\begin{align*}
    w_{n,1}^{-1}\sum_{x \in (x_{0,1}, x_{0,1}+rw_{n,1}] \cap \calX_1}\vepsbar_x \rightsquigarrow \sigma\sum_{k=1}^{r}  Z_k, \quad w_{n,1}^{-1}\sum_{x \in (x_{0,1}-\ell w_{n,1},x_{0,1}] \cap \calX_1}\vepsbar_x \rightsquigarrow \sigma\sum_{k=1}^{\ell}  Z_{-k},
\end{align*}
where $\{Z_k\}_{k \in \Z_{\neq 0}}$ are i.i.d.\ standard normal random variables. Let $S$ be the associated two-sided random walk, together with its piecewise linear interpolation, as defined in Section~\ref{limit-them-section}. Then we have
\begin{align*}
    R_n &\rightsquigarrow \max_{r \in [0,c] \cap \Z}\f{-f_1^{*\prime}(x_{0,1})r(r+1)}{2}- \sigma S(r)+ rt,\\
    L_n &\rightsquigarrow \max_{\ell \in [1,c] \cap \Z} \f{-f_1^{*\prime}(x_{0,1})(\ell-1)\ell}{2}-\sigma S(-\ell) - \ell t.
\end{align*}

Hence, by the independence of $R_n$ and $L_n$,
\begin{align*}
    &\P\left\{ w_{n,1}^{-1}(\fhat^{or}_{1,c}(x_{0,1}) -f_1^*(x_{0,1})) \le t \right\} \\
    &\rightarrow \P\bigg\{ \min_{r \in [0,c] \cap \Z} \max_{\ell \in [1, c] \cap \Z } \sigma S(r) - \sigma S(-\ell) \\
    &\quad\quad\quad + \f{f_1^{*\prime}(x_{0,1})r(r+1)}{2} - \f{f_1^{*\prime}(x_{0,1})\ell(\ell-1)}{2} - t(r+\ell) \le 0\bigg\}.
\end{align*}
Here the convergence of the probabilities holds because $R_n$ and $L_n$ are independent for every $n$, and the limiting random variable inside the probability is a min--max over the finite index set $([0,c]\cap\Z)\times([1,c]\cap\Z)$ of Gaussian random variables. The min--max value always equals one of finitely many continuously distributed random variables, so $0$ is a continuity point of its distribution function.

By Propositions~\ref{large-localization-prop} and \ref{local-limit-critical} and the Lemma~\ref{delocalization-lemma}, we have
\begin{align*}
    &\P\left\{ w_{n,1}^{-1}(\fhat^{or}_{1}(x_{0,1}) -f_1^*(x_{0,1})) \le t \right\} \\
    &\rightarrow \P\left\{\min_{r \in \Z_{\ge 0}} \max_{\ell \in \Z_{\ge 1} } \f{ \sigma S(r) - \sigma S(-\ell)}{r+\ell}+\f{f_1^{*\prime}(x_{0,1})(r-\ell+1)}{2}\le t\right\}  \\
    &= \P\left\{ \mathbb{G}_{f_1^{*\prime}(x_{0,1})} \le t \right\},
\end{align*}
where $\mathbb{G}_{\alpha}$ is defined in \eqref{gcm-rw-def}.
\end{proof}

\subsection{Asymptotic independence under the fixed lattice design}
In this section, we prove asymptotic joint independence of $(\fhat_1(x_{0,1}), \ldots, \fhat_d(x_{0,d}))$. By Lemma 3.1 of \citet{guntuboyina2018nonparametric} and the discussion at the beginning of Section~\ref{sec-proof-lim-them}, it suffices to prove asymptotic joint independence for the oracle LSEs $(\fhat_1^{or}(x_{0,1}), \ldots, \fhat_d^{or}(x_{0,d}))$ characterized by \eqref{min-max-formula}.
By the min--max representation in \eqref{min-max-formula}, the estimators $\{\fhat_j^{or}\}_{j=1}^d$ interact only through the overlap of the strips over which the local averages are computed. For dimension $j$, the estimator $\fhat_j^{or}(x_{0,j})$ averages observations over the strip $\{x = (x_1,\dots, x_d)^T \in \mathcal{X}: x_j \in (x_{0,j} - \ellhat_j w_{n,j}, x_{0,j}+\rhat_j w_{n,j}]\}$, and this strip intersects the corresponding strips for $\fhat_l^{or}$ with $l \neq j$. We compare a localized version of $\fhat_j^{or}$, defined in \eqref{eqn-local-fhat-def} below, with a modified estimator that removes the overlap region and show that the resulting discrepancy is asymptotically negligible. For this comparison, we need localization bounds on both strip endpoints: lower bounds ensure that the local averages still involve sufficiently many design points after the overlap is removed, while upper bounds keep the overlap region under control. Accordingly, we strengthen the one-sided localization result for $\fhat_j^{or}$ in Proposition~\ref{large-localization-prop} and establish a two-sided localization result.

Recall $ \Ybar^{(j)}_{x_{0,j}} = n^{-\sum_{k\neq j}\beta_k}\sum_{i=1}^n Y_{i}\ind\{x_{i,j} = x_{0,j}\}$ as defined in \eqref{Ybar-def} and $ F_{n,j}\big(i/n_j\big) = \f{1}{n_j} \sum_{k=1}^{i}\Ybar_{x_{k,j}}^{(j)}$ for $i=1,\ldots, n_j$. Define a localized oracle estimator $\fhat_{j,c,\gamma}^{or}(x_{0,j})$ as follows. Fix $c>0$. When $\beta_j>1/3$, also fix $\gamma>0$.
\begin{align}\label{eqn-local-fhat-def}
    \fhat_{j,c,\gamma}^{or}(x_{0,j})=
    \begin{cases}
        \min_{c^{-\gamma} \le r \le c}\max_{c^{-\gamma} \le \ell \le c}\f{F_{n,j}(x_{0,j}+rw_{n,j}) - F_{n,j}(x_{0,j} - \ell w_{n,j})}{(r+\ell)w_{n,j}}, & \beta_j> 1/3 \\
        \min_{0 \le r \le c}\max_{0 < \ell \le c}\f{F_{n,j}(x_{0,j}+rw_{n,j}) - F_{n,j}(x_{0,j} - \ell w_{n,j})}{(r+\ell)w_{n,j}}, & \beta_j =1/3\\
        \Ybar_{x_{0,j}}^{(j)}, & \beta_j < 1/3
    \end{cases}
\end{align}
For $\beta_j \le 1/3$, this reduces to the one-sided localization in \eqref{f-oracle-local} with the index $1$ replaced by $j$. The additional lower truncation $r,\ell\ge c^{-\gamma}$ is imposed only when $\beta_j>1/3$. Throughout this subsection, $\gamma$ is fixed and is suppressed in the subsequent vector and endpoint notation. For $\beta_j\le1/3$, the index $\gamma$ is vacuous.
The next proposition shows that this localization does not change the oracle estimator asymptotically.

\begin{proposition}\label{small-localization-prop}
    Let $\fhat_j^{or}(x_{0,j})$ be the oracle estimator characterized by \eqref{f-oracle} with the index $1$ replaced by $j$, and let $\fhat_{j,c,\gamma}^{or}(x_{0,j})$ be defined in \eqref{eqn-local-fhat-def}. Then
    {
    \[
        \lim_{c \rightarrow \infty}\limsup_{n \rightarrow  \infty}\P\Big\{\fhat_j^{or}(x_{0,j}) \ne \fhat_{j,c,\gamma}^{or}(x_{0,j})\Big\} = 0.
    \]
    }
\end{proposition}
\begin{proof}
    Without loss of generality, take $j=1$. Let $\rhat_1$ and $\ellhat_1$ be the scaled interval widths in \eqref{f-oracle}. When $\beta_1 < 1/3$, Proposition~\ref{large-localization-prop} already gives $\P\{\rhat_1 = 0,\ \ellhat_1 \le (n_1w_{n,1})^{-1}\}\to 1$. On this event the increment of $F_{n,1}$ in \eqref{f-oracle} spans only the strip at $x_{0,1}$, so $\fhat_1^{or}(x_{0,1})=\Ybar_{x_{0,1}}$, which coincides with the localized definition. Because Proposition~\ref{large-localization-prop} similarly covers $\beta_1 = 1/3$, it remains to consider the case $\beta_1 > 1/3$.

    Proposition~\ref{large-localization-prop} already yields the upper bounds $\rhat_1 \vee \ellhat_1 \le c$ with high probability for large $c$. 
Moreover, on the event $E_c := \{c^{-\gamma} \le \rhat_1 \wedge \ellhat_1,\ \rhat_1 \vee \ellhat_1 \le c\}$ we have $\fhat_1^{or}(x_{0,1}) = \fhat_{1,c,\gamma}^{or}(x_{0,1})$. Indeed, write $A(r,\ell)$ for the objective in \eqref{f-oracle}, and recall the well-known switch property of isotonic least squares: the min--max in \eqref{f-oracle} equals the corresponding max--min, and $(\rhat_1,\ellhat_1)$ attains both values (see, e.g., Section~3.3 of \citet{groeneboom2014nonparametric}). Consequently,
\[
    A(\rhat_1, \ell) \ \le\ A(\rhat_1, \ellhat_1) = \fhat_1^{or}(x_{0,1}) \ \le\ A(r, \ellhat_1)
    \qquad \text{for all } r \ge 0,\ \ell > 0.
\]
On $E_c$ the pair $(\rhat_1, \ellhat_1)$ is feasible for the restricted problem \eqref{eqn-local-fhat-def}, whence
\[
    \fhat_1^{or}(x_{0,1})
    \ \le\ \min_{c^{-\gamma}\le r \le c} A(r, \ellhat_1)
    \ \le\ \fhat_{1,c,\gamma}^{or}(x_{0,1})
    \ \le\ \max_{c^{-\gamma}\le \ell \le c} A(\rhat_1, \ell)
    \ \le\ \fhat_1^{or}(x_{0,1}),
\]
i.e., $\fhat_{1,c,\gamma}^{or}(x_{0,1}) = \fhat_1^{or}(x_{0,1})$ on $E_c$. Therefore, it is enough to prove that
    \[
        \lim_{c \rightarrow \infty}\limsup_{n \rightarrow \infty}\P\{\rhat_1 < c^{-\gamma} \text{ or } \ellhat_1 < c^{-\gamma}\} = 0.
    \]
    By symmetry, it suffices to rule out the event $\{\rhat_1 \le c^{-\gamma}\}$. Fix $0 < b < \gamma$, and write $n_{-1} = n^{\sum_{l \neq 1}\beta_l}$. Let $d_n=(n_1w_{n,1})^{-1}$ denote the scaled lattice spacing, so every feasible positive width belongs to $d_n\Z_{\ge1}$. On the event $\{\rhat_1 \le c^{-\gamma}\}$, the min--max formula gives, for all sufficiently large $n$,
    \begin{align}
        &\fhat_1^{or}(x_{0,1}) - f_1^*(x_{0,1}) \nonumber\\
        &\ge \max_{\substack{d_n \le \ell \le c^{-b}\\ \ell\in d_n\Z}}\left[\overline{f_1^*}\big|_{(x_{0,1} -\ell w_{n,1},\, x_{0,1}+\rhat_1w_{n,1}]}+\vepsbar\big|_{(x_{0,1} - \ell w_{n,1},\, x_{0,1}+\rhat_1w_{n,1}]}\right] - f_1^*(x_{0,1})\nonumber\\
        &\ge f_1^*(x_{0,1} - c^{-b}w_{n,1}) - f_1^*(x_{0,1}) + \max_{\substack{d_n \le \ell \le c^{-b}\\ \ell\in d_n\Z}}\vepsbar\big|_{(x_{0,1} - \ell w_{n,1},\, x_{0,1}+\rhat_1w_{n,1}]}\nonumber\\
        &\ge -C_1c^{-b}w_{n,1} + \max_{\substack{d_n \le \ell \le c^{-b}\\ \ell\in d_n\Z}}\vepsbar\big|_{(x_{0,1} - \ell w_{n,1},\, x_{0,1}+\rhat_1 w_{n,1}]}. \label{small-dev}
    \end{align}
    Decompose the stochastic term into its left and right contributions by defining
    \begin{align*}
        M_1(\ell) &=\f{n_1}{n}\sum_{i=1}^n  \veps_{i}\ind\{ x_{i,1} \in (x_{0,1} -\ell w_{n,1}, x_{0,1}] \},\\
        M_2 &=\f{n_1}{n}\sum_{i=1}^n  \veps_{i}\ind\{ x_{i,1} \in (x_{0,1}, x_{0,1}+\rhat_1 w_{n,1}] \}.
    \end{align*}
    First control the short right-hand contribution. Since $r \mapsto \sum_{i=1}^n \veps_{i}\ind\{ x_{i,1} \in (x_{0,1}, x_{0,1}+r w_{n,1}] \}$ is a martingale in $r$ with variance $\sigma^2$ times the number of design points in the window, Doob's maximal inequality (as in the proof of Lemma~\ref{error-control}) yields, on the event $\{\rhat_1 \le c^{-\gamma}\}$ and for sufficiently large $n$,
    \begin{align}\label{eqn-13}
        |M_2| &\le \f{1}{n_{-1}}\sup_{0 \le r \le c^{-\gamma}}\bigg|\sum_{i=1}^n  \veps_{i}\ind\{ x_{i,1} \in (x_{0,1}, x_{0,1}+r w_{n,1}] \}\bigg|\nonumber\\
        & = O_p(\sigma w_{n,1}^{2}n_1 c^{-\gamma/2}).
    \end{align}

    Next, control the left-hand fluctuation. Set $M_1(0)=0$ and extend $M_1(\ell)$ by linear interpolation between consecutive grid widths. Since $\beta_1>1/3$ implies $w_{n,1}=n^{-1/3}$, we have
    \[
        n_1^{-1}w_{n,1}^{-2}M_1(\ell)
        = w_{n,1}\sum_{i=1}^n  \veps_{i}\ind\{ x_{i,1} \in (x_{0,1} -\ell w_{n,1}, x_{0,1}] \},
    \]
    and Lemma~\ref{weak-convergence-lemma} yields
    \begin{align*}
        n_1^{-1}w_{n,1}^{-2} M_1 \rightsquigarrow \sigma B,
    \end{align*}
    where $B$ is a standard Brownian motion. Fix $\delta>0$ and choose $t_{\delta}>0$ such that $\P(Z > t_{\delta}) \ge 1/2 - \delta/4$, where $Z\sim N(0,1)$. Let $\rho = \sigma c^{-b/2}t_{\delta}$. Since $M_1(0)=0$, $M_1$ is linear on $[0,d_n]$, and $\rho>0$, adjoining $[0,d_n)$ does not change the event in the first line below. The continuous mapping theorem and continuity of the distribution of the Brownian maximum therefore give, for all sufficiently large $n$,
    \begin{align*}
        \P \left\{ \max_{d_n \le \ell \le c^{-b}}n_1^{-1}w_{n,1}^{-2} M_1(\ell) \le \rho \right\} &\le \P\left\{  \max_{0 \le \ell \le c^{-b} } \sigma B(\ell) \le \rho \right\}+\f{\delta}{2}\\
        & = 1- \P\left\{ \max_{0 \le \ell \le c^{-b}}B(\ell) >  c^{-b/2}t_{\delta}\right\}+\f{\delta}{2}\\
        &\stackrel{(a)}{=}1 - 2\P\left\{ B(1) >  t_{\delta}\right\}+\f{\delta}{2}\\
        & \le \delta,
    \end{align*}
    where $(a)$ is by the reflection principle. Hence, with probability at least $1-\delta$,
    \begin{align}
        \max_{d_n \le \ell \le c^{-b}}M_1(\ell) > n_1w_{n,1}^2 \rho.\label{eqn-32}
    \end{align}

    Finally, bound the number of coordinate levels in the averaging
    interval. For $\rhat_1\le c^{-\gamma}$ and every feasible
    $\ell\in[d_n,c^{-b}]\cap d_n\Z$, the half-open convention gives
    \begin{align}
        \big|\{x_1 \in \calX_1 : x_1 \in
        (x_{0,1} - \ell w_{n,1}, x_{0,1}+\rhat_1 w_{n,1}]\}\big|
        =(\ell+\rhat_1)n_1w_{n,1}
        \le(c^{-b}+c^{-\gamma})n_1w_{n,1}. \label{eqn-33}
    \end{align}
    Combining \eqref{small-dev}, \eqref{eqn-13}, \eqref{eqn-32}, and \eqref{eqn-33}, and using $0<b<\gamma$, we obtain on an event with probability at least $1-\delta$,
    \begin{align*}
        \fhat_1^{or}(x_{0,1}) - f_1^*(x_{0,1}) &\ge \f{n_1w_{n,1}^2 \sigma \bigl( c^{-b/2}t_{\delta} - C_2 c^{-\gamma/2} \bigr)}{(c^{-b}+c^{-\gamma})n_1w_{n,1}}-C_1c^{-b}w_{n,1}\\
        &\ge w_{n,1} C_{\delta, \sigma}\bigl( c^{b/2} - c^{-b} \bigr),
    \end{align*}
    where $C_{\delta,\sigma}>0$ does not depend on $c$ or $n$. Proposition~\ref{rate-cvgc} states that $\fhat_1^{or}(x_{0,1}) - f_1^*(x_{0,1}) = O_p(w_{n,1})$, so the above lower bound can hold only with vanishing probability as $c \to \infty$. Therefore,
    \[
        \lim_{c \rightarrow \infty}\limsup_{n \rightarrow \infty}\P\{\rhat_1 \le c^{-\gamma}\} = 0.
    \]
    The same argument applies to $\ellhat_1$.
\end{proof}

We now prove the joint independence of $\fhat_j^{or}$ for $j=1, \ldots, d$.
\begin{proof}[Proof of Theorem \ref{theorem-lattice}, asymptotic independence]
Recall from Assumption~\ref{fixed-design-assumption} that the fixed lattice design satisfies $\mathcal{X}=\prod_{l=1}^d \calX_l$, where $\calX_l=\{x_{1,l},\ldots,x_{n_l,l}\}\subset[0,1]$. For a design point $x_i=(x_{i,1},\ldots,x_{i,d}) \in \mathcal{X}$, write $x_{i,-j}=(x_{i,1},\ldots,x_{i,j-1},x_{i,j+1},\ldots,x_{i,d}) \in \R^{d-1}$. Given $c>0$ and a design point $x_0=(x_{0,1},\ldots,x_{0,d})\in\mathcal{X}$, define

\begin{align*}
    \mathcal{S}_{-j,c}
    &:= \Big\{x_{-j} \in \prod_{l\neq j}\calX_l :\ x_l \in [x_{0,l}-cw_{n,l},\,x_{0,l}+cw_{n,l}]\ \text{ for some } l \neq j\Big\}\subset \R^{d-1},\\
    \vepsbar_{x_j}^{(-j)}
    &:= \f{n_j}{n}\sum_{i=1}^n \veps_i \ind\{x_{i,j}=x_j,\; x_{i,-j}\notin \mathcal{S}_{-j,c}\}, \qquad x_j \in \calX_j,\\
    \Ytil_{x_j}^{(j)}
    &:= f_j^*(x_j)+\vepsbar_{x_j}^{(-j)}, \qquad x_j \in \calX_j.
\end{align*}

Thus, $\vepsbar_{x_j}^{(-j)}$ averages the noise in the $j$th strip after discarding every design point for which \emph{some} coordinate $l\neq j$ falls in its localization window $[x_{0,l}-cw_{n,l}, x_{0,l}+cw_{n,l}]$, and $\Ytil_{x_j}^{(j)}$ is the corresponding pseudo-response. Using $\{\Ytil_{x_j}^{(j)}\}_{x_j\in\calX_j}$, define the localized estimator obtained by removing this overlap region:
\begin{align}\label{eqn-local-ftil-def}
    \ftil_{j,c}(x_{0,j})=
    \begin{cases}
        \min\limits_{c^{-\gamma} \le r\le c}\max\limits_{c^{-\gamma} \le \ell \le c}\f{\sum_{x_j \in \calX_j}\Ytil_{x_{j}}^{(j)}\ind\{x_j \in (x_{0,j} - \ell w_{n,j} , x_{0,j}+rw_{n,j}] \}}{|\{ x_j \in \calX_j: x_j \in (x_{0,j} - \ell w_{n,j} , x_{0,j}+rw_{n,j}] \}|}, & \beta_j >1/3\\
        \min\limits_{0 \le r\le c}\max\limits_{0 < \ell \le c}\f{\sum_{x_j \in \calX_j}\Ytil_{x_{j}}^{(j)}\ind\{x_j \in (x_{0,j} - \ell w_{n,j} , x_{0,j}+rw_{n,j}]\}}{|\{ x_j \in \calX_j: x_j \in (x_{0,j} - \ell w_{n,j} , x_{0,j}+rw_{n,j}] \}|}, &\beta_j =1/3\\
        \Ytil_{x_{0,j}}^{(j)}, & \beta_j <1/3
    \end{cases}
\end{align}
Every error entering $\ftil_{j,c}(x_{0,j})$ satisfies $x_{i,l}\notin[x_{0,l}-cw_{n,l}, x_{0,l}+cw_{n,l}]$ for all $l\neq j$, whereas every error entering $\fhat_{l,c,\gamma}^{or}(x_{0,l})$ satisfies $x_{i,l}\in[x_{0,l}-cw_{n,l}, x_{0,l}+cw_{n,l}]$. The two collections of errors are therefore disjoint, and since the errors are independent, $\ftil_{j,c}(x_{0,j})$ is independent of the vector $\{\fhat_{l,c,\gamma}^{or}(x_{0,l})\}_{l \neq j}$.

Given a design point $x_0 \in \mathcal{X}$, define $\bff^*(x_0) = (f_1^*(x_{0,1}), \ldots, f_d^*(x_{0,d}))^T$ and $\hat{\bff}^{or}(x_0)  = (\fhat_{1}^{or}(x_{0,1}), \ldots,\fhat_{d}^{or}(x_{0,d}) )^T$. $\hat{\bff}_c^{or}(x_{0})$ and $\tilde{\bff}_c(x_0)$ can be defined in the same fashion. Furthermore, define
\begin{align*}
    \Delta_{n}^{or} &=  w_{n}^{-1}\odot(\hat{\bff}^{or}(x_{0}) -\bff^*(x_{0})),\\
    \Delta_{n,c}^{or} &= w_{n}^{-1} \odot (\hat{\bff}_c^{or}(x_0) - \bff^*(x_0)),\\
    \tilde{\Delta}_{n,c} &= w_{n}^{-1} \odot (\tilde{\bff}_c(x_0) - \bff^*(x_0)),
\end{align*}
where $w_{n}^{-1} = (w_{n,1}^{-1}, {\ldots}, w_{n,d}^{-1})$ and $\odot$ represents the elementwise product.

For each $j$, define the good event
\begin{equation}\label{good-event-H}
    H_j := \Big\{\fhat_j^{or}(x_{0,j}) = \fhat_{j,c,\gamma}^{or}(x_{0,j})\Big\}.
\end{equation}
By Proposition~\ref{small-localization-prop}, we have $\lim_{c \rightarrow \infty}\limsup_{n \rightarrow \infty}\P(H_j^c) = 0$.

Let $\Delta_{n,j}^{or}$, $\Delta_{n,c,j}^{or}$, and $\tilde{\Delta}_{n,c,j}$ denote the $j$th coordinates of $\Delta_{n}^{or}$, $\Delta_{n,c}^{or}$, and $\tilde{\Delta}_{n,c}$, respectively. We first prove joint asymptotic independence for $\Delta_{n,c}^{or}$. The result for $\Delta_n^{or}$ then follows because localization holds with high probability.

In order to show the asymptotic independence of $\Delta_{n,c}^{or}$, it suffices to show that for any $t \in \R^d$,
\begin{align}\label{joint-indep-eqn1}
    \lim_{n \rightarrow \infty} \bigg|\E \exp(i t^T \Delta_{n, c}^{or}) - \prod_{k=1}^d\E \exp(it_k\Delta_{n,c, k}^{or})\bigg| = 0.
\end{align}

For $m=1,\ldots,d$, define
\[
    A_{n,c}^{(m)}(t)
    := \bigg| \E\exp\Big(i\sum_{l=m}^d t_l\Delta_{n,c,l}^{or}\Big) - \prod_{l=m}^d\E\exp\bigl(it_l\Delta_{n,c,l}^{or}\bigr) \bigg|,
\]
with the convention $A_{n,c}^{(d+1)}(t)=0$. For each $m$, add and subtract the term with $\tilde{\Delta}_{n,c,m}$. Since $\tilde{\Delta}_{n,c,m}$ is independent of $\{\Delta_{n,c,l}^{or}\}_{l>m}$, and since $\E|e^{itX}-e^{itY}| \le |t|\E|X-Y|$ for any real-valued random variables $X$ and $Y$, we obtain
\[
    A_{n,c}^{(m)}(t)
    \le 2|t_m|\E\big|\Delta_{n,c,m}^{or}-\tilde{\Delta}_{n,c,m}\big| + A_{n,c}^{(m+1)}(t).
\]
Iterating this bound from $m=1$ to $m=d$ gives
\begin{align}\label{eqn-3}
    \bigg| \E\Big[ \exp\bigl(it^T\Delta_{n,c}^{or} \bigr) \Big] - \prod_{j=1}^d\E\Big[\exp\bigl( it_j\Delta_{n,c,j}^{or} \bigr)\Big] \bigg| \le 2\sum_{j=1}^d |t_j|\E\big|\Delta_{n,c,j}^{or} - \tilde{\Delta}_{n,c,j}\big|.
\end{align}

For $l = 1,\ldots,d$, let $N_l := \big|[x_{0,l}-cw_{n,l},\, x_{0,l}+cw_{n,l}] \cap \calX_l\big|$ be the number of design points of $\calX_l$ in the $l$th localization window, and let
\[
    m_{n,c,j} := \Big|\Big\{x \in \calX :\ x_j \in [x_{0,j}-cw_{n,j},\, x_{0,j}+cw_{n,j}],\ x_{-j} \in \mathcal{S}_{-j,c}\Big\}\Big|
\]
be the number of design points discarded from the $j$th localization slab. Since $\calX_l$ has spacing $1/n_l$ and $x_{0,l}\in\calX_l$, we have $1 \le N_l \le 2cw_{n,l}n_l + 1 \le 4c\, n_l \rho_l$ for $c \ge 1$, where
\[
    \rho_l := w_{n,l} \vee n_l^{-1} = n^{-(\beta_l \wedge 1/3)}
 \]
by \eqref{rate-cvgc-def}. A union bound over the coordinate $l\neq j$ responsible for membership in $\mathcal{S}_{-j,c}$ gives
\begin{align}\label{eqn-14}
    m_{n,c,j} \le \sum_{l \neq j} N_j N_l \prod_{k \neq j,l} n_k
    \le 16c^2\, n\, \rho_j \sum_{l \neq j} \rho_l
    \le C_{c,d}\, n\, \rho_j\, n^{-\kappa_j},
    \qquad \kappa_j := \min_{l \neq j}(\beta_l \wedge 1/3) > 0.
\end{align}

Let $\rhat_{j,c}$ and $\ellhat_{j,c}$ be the scaled minimizing and maximizing widths in \eqref{eqn-local-fhat-def}, and let $\widetilde r_{j,c}$ and $\widetilde\ell_{j,c}$ be the corresponding widths in \eqref{eqn-local-ftil-def}. Let $\ell_{j,+}$ be a maximizer of the objective in \eqref{eqn-local-fhat-def} at $r=\widetilde r_{j,c}$, and let $\ell_{j,-}$ be a maximizer of the objective in \eqref{eqn-local-ftil-def} at $r=\rhat_{j,c}$. Define

\begin{align*}
    \mathcal{I}_{j,+}
    &:= \{x_j \in \calX_j : x_j \in (x_{0,j} - \ell_{j,+}w_{n,j},\, x_{0,j}+\widetilde r_{j,c}w_{n,j}]\},\\
    \mathcal{I}_{j,-}
    &:= \{x_j \in \calX_j : x_j \in (x_{0,j} - \ell_{j,-}w_{n,j},\, x_{0,j}+\rhat_{j,c}w_{n,j}]\}.
\end{align*}

Writing $n_{-j}:=n/n_j=\prod_{l\neq j} n_l$, we have $n_j/n = 1/n_{-j}$. Since $\widetilde r_{j,c}$ is feasible for the outer minimization in \eqref{eqn-local-fhat-def} and $\ell_{j,+}$ attains the corresponding inner maximum, and since $\ftil_{j,c}(x_{0,j})$ is bounded below by the objective of \eqref{eqn-local-ftil-def} evaluated at $(\widetilde r_{j,c}, \ell_{j,+})$, we obtain
\begin{align*}
    \fhat_{j,c,\gamma}^{or}(x_{0,j}) - \ftil_{j,c}(x_{0,j})
    &\le \f{\sum_{x_j \in \mathcal{I}_{j,+}}\bigl(\vepsbar_{x_j}^{(j)}-\vepsbar_{x_j}^{(-j)}\bigr)}{|\mathcal{I}_{j,+}|}\\
    &= \f{\sum_{i=1}^n\veps_i\ind\{x_{i,j} \in \mathcal{I}_{j,+},\, x_{i,-j} \in \mathcal{S}_{-j,c}\}}{n_{-j}|\mathcal{I}_{j,+}|}
\end{align*}
and, symmetrically, evaluating \eqref{eqn-local-fhat-def} at $(\rhat_{j,c}, \ell_{j,-})$ and using that $\ell_{j,-}$ attains the inner maximum in \eqref{eqn-local-ftil-def} at $r=\rhat_{j,c}$,
\begin{align}\label{neg-lower-bound}
    \fhat_{j,c,\gamma}^{or}(x_{0,j}) - \ftil_{j,c}(x_{0,j})
    \ge \f{ \sum_{i=1}^n\veps_i\ind\{ x_{i,j} \in \mathcal{I}_{j,-},\, x_{i,-j} \in \mathcal{S}_{-j,c} \} }{n_{-j}|\mathcal{I}_{j,-}|}.
\end{align}
We now show that $\E|\Delta_{n,c,j}^{or} - \tilde{\Delta}_{n,c,j}| \to 0$ as $n \rightarrow \infty$. We begin with the positive part.
For $\beta_j\ge 1/3$ we have
\begin{align*}
    \E\left(\Delta_{n,c,j}^{or} - \tilde{\Delta}_{n,c,j}\right)_+
    \le w_{n,j}^{-1} \E\left(\f{\left|\sum_{i=1}^n\veps_i\ind\{ x_{i,j} \in \mathcal{I}_{j,+},\, x_{i,-j} \in \mathcal{S}_{-j,c} \}\right|}{n_{-j}|\mathcal{I}_{j,+}|}\right).
\end{align*}
Since $|\mathcal{I}_{j,+}| \ge \lfloor c^{-\gamma}w_{n,j}n_j \rfloor$ for $\beta_j >1/3$ by the definition of $\widetilde r_{j,c}$, and $|\mathcal{I}_{j,+}|\ge 1$ for $\beta_j = 1/3$, we obtain
\begin{align}
    \E\left(\Delta_{n,c,j}^{or} - \tilde{\Delta}_{n,c,j}\right)_+
    &\le \f{C_{c, \gamma}w_{n,j}^{-2}}{n_{-j}n_j}\E\bigg|\sum_{i=1}^n\veps_i \ind\{ x_{i,j} \in \mathcal{I}_{j,+},\, x_{i,-j} \in \mathcal{S}_{-j,c}\}\bigg| \nonumber\\
    &\le \f{C_{c, \gamma}w_{n,j}^{-2}}{n}\E \bigg[ \max_{1 \le m \le m_{n,c,j}}\Big|\sum_{l=1}^m\veps_l \Big| \bigg]\nonumber\\
    &\le C_{c, d, \gamma}\sigma n^{-\kappa_j/2},\label{eqn-15}
\end{align}
where $C_{c, \gamma}$ and $C_{c, d, \gamma}$ are constants. For the second inequality, enumerate the discarded design points strip by strip, moving away from $x_{0,j}$ separately to the right and to the left of $x_{0,j}$. Each one-sided sum is then a prefix sum of a fixed enumeration of at most $m_{n,c,j}$ error terms, and the resulting factor $2$ is absorbed into the constant. The last inequality follows from Doob's inequality together with \eqref{eqn-14} and $\rho_j = n^{-1/3}$ for $\beta_j \ge 1/3$, since $w_{n,j}^{-2}n^{-1}\sqrt{C_{c,d}\,n\,\rho_j\,n^{-\kappa_j}} = C_{c,d}^{1/2}\, n^{2/3-1+1/2-1/6-\kappa_j/2} = C_{c,d}^{1/2}\, n^{-\kappa_j/2}$. When $\beta_j < 1/3$, both estimators average the single strip at $x_{0,j}$, and the number of discarded points in this strip is $|\mathcal{S}_{-j,c}| \le \sum_{l\neq j}N_l\prod_{k\neq j,l}n_k \le 4c\,n_{-j}\sum_{l\neq j}\rho_l \le C_{c,d}\,n_{-j}\,n^{-\kappa_j}$. Hence, we have
\begin{align}
    \E\left(\Delta_{n,c,j}^{or} - \tilde{\Delta}_{n,c,j}\right)_+ &\le \f{w_{n,j}^{-1}}{n_{-j}}\E \bigg|\sum_{i=1}^n\veps_i\ind\{  x_{i,j} = x_{0,j},\, x_{i,-j} \in \mathcal{S}_{-j,c} \}\bigg|\nonumber\\
    & \le \f{\sigma w_{n,j}^{-1}\sqrt{|\mathcal{S}_{-j,c}|}}{n_{-j}}\nonumber\\
    & \le C_{c,d}\,\sigma\, \f{w_{n,j}^{-1}}{\sqrt{n_{-j}}}\, n^{-\kappa_j/2}\nonumber\\
    &= C_{c,d}\,\sigma\, n^{-\kappa_j/2},\label{eqn-16}
\end{align}
where the last equality uses $w_{n,j}^{-1}=n^{(1-\beta_j)/2}=\sqrt{n_{-j}}$ for $\beta_j<1/3$. In both cases the bound goes to $0$ as $n \rightarrow \infty$.
The lower bound \eqref{neg-lower-bound} is handled in the same way, so \eqref{eqn-15} and \eqref{eqn-16} imply
\begin{align*}
    \E|\Delta_{n,c,j}^{or} - \tilde{\Delta}_{n,c,j}| \rightarrow  0, \quad \text{as } n \rightarrow \infty.
\end{align*}
Combined with \eqref{eqn-3}, we have
\begin{align}\label{eqn-5}
    \bigg| \E\Big[ \exp\bigl(it^T\Delta_{n,c}^{or} \bigr) \Big] - \prod_{j=1}^d\E\Big[\exp\bigl( it_j\Delta_{n,c,j}^{or} \bigr)\Big] \bigg| \rightarrow 0.
\end{align}

We now transfer \eqref{eqn-5} from the localized oracle estimators to the original oracle estimators. Since $\Delta_{n}^{or}=\Delta_{n,c}^{or}$ on the event $\cap_{j=1}^d H_j$ in \eqref{good-event-H}, it is enough to control the marginal characteristic functions and the joint characteristic function outside this event.

First, for each $j$,
\[
    \Big|\E\exp\bigl(it_j\Delta_{n,c,j}^{or}\bigr)-\E\exp\bigl(it_j\Delta_{n,j}^{or}\bigr)\Big|
    \le 2\P(H_j^c).
\]
Applying this bound iteratively yields
\begin{align*}
    \bigg|\prod_{j=1}^d\E\Big[\exp\bigl( it_j\Delta_{n,c,j}^{or} \bigr)\Big]  - \prod_{j=1}^d\E\Big[\exp\bigl( it_j\Delta_{n,j}^{or} \bigr)\Big]  \bigg| \le 2\sum_{j=1}^d \P(H_j^c).
\end{align*}
Let $\zeta =  \exp\bigl(it^T\Delta_{n}^{or} \bigr) -\exp\bigl(it^T\Delta_{n,c}^{or} \bigr)$. Since $|\zeta| \le 2$, we have
\begin{align*}
    \left| \E\zeta\right| &\le \left| \E\big[ \zeta \ind\{\cap_{j=1}^d H_j\}\big]  \right| + \left| \E\big[\zeta \ind\{\cup_{j=1}^d H_j^c\}\big]  \right| \\
    &\le 2\sum_{j=1}^d \P(H_j^c).
\end{align*}
Combining the above inequalities, we obtain
\begin{align}
    & \bigg| \E\Big[ \exp\bigl(it^T\Delta_{n}^{or} \bigr) \Big] - \prod_{j=1}^d\E\Big[\exp\bigl( it_j\Delta_{n,j}^{or} \bigr)\Big] \bigg|\nonumber\\
     &\le | \E\zeta|+\Big|\E\Big[ \exp\bigl(it^T\Delta_{n,c}^{or} \bigr)\Big]- \prod_{j=1}^d\E\Big[\exp\bigl( it_j\Delta_{n,c,j}^{or} \bigr)\Big]\Big|\nonumber\\
     &\quad +\bigg|\prod_{j=1}^d\E\Big[\exp\bigl( it_j\Delta_{n,c,j}^{or} \bigr)\Big]  - \prod_{j=1}^d\E\Big[\exp\bigl( it_j\Delta_{n,j}^{or} \bigr)\Big]  \bigg|\nonumber\\
     &\le \bigg|\E\Big[ \exp\bigl(it^T\Delta_{n,c}^{or} \bigr) \Big] - \prod_{j=1}^d\E\Big[\exp\bigl( it_j\Delta_{n,c,j}^{or} \bigr)\Big]\bigg| + 4\sum_{j=1}^d \P(H_j^c).\label{eqn-17}
\end{align}
Given $\varepsilon>0$, choose $c$ large enough that $\limsup_{n \rightarrow \infty}\sum_{j=1}^d\P(H_j^c)<\varepsilon$. For this fixed $c$, let $n \to \infty$ in \eqref{eqn-17} and use \eqref{eqn-5}. We obtain
\begin{align*}
    \limsup_{n \rightarrow \infty}\bigg| \E\Big[ \exp\bigl(it^T\Delta_{n}^{or} \bigr) \Big] - \prod_{j=1}^d\E\Big[\exp\bigl( it_j\Delta_{n,j}^{or} \bigr)\Big] \bigg| \le 4\varepsilon.
\end{align*}
Since $\varepsilon>0$ is arbitrary,
\begin{align*}
    \lim_{n \rightarrow \infty}\bigg| \E\Big[ \exp\bigl(it^T\Delta_{n}^{or} \bigr) \Big] - \prod_{j=1}^d\E\Big[\exp\bigl( it_j\Delta_{n,j}^{or} \bigr)\Big] \bigg| = 0.
\end{align*}
Let $D_1,\ldots,D_d$ be mutually independent random variables with
$D_j\stackrel{d}{=}\mathbb{D}_{\beta_j,f_j^{*\prime}(x_{0,j})}$. The three
marginal convergence results proved above give
$\Delta_{n,j}^{or}\rightsquigarrow D_j$ for every $j$. Hence,
$\prod_{j=1}^d\E\exp(it_j\Delta_{n,j}^{or})$ converges to
$\prod_{j=1}^d\E\exp(it_jD_j)$, which is the characteristic function of
$(D_1,\ldots,D_d)^T$. The preceding display and L\'evy's continuity theorem
therefore yield the joint convergence of $\Delta_n^{or}$. Finally,
$w_{n,j}^{-1}\{\fhat_j(x_{0,j})-\fhat_j^{or}(x_{0,j})\}
=O_p(w_{n,j}^{-1}n^{-1/2})=o_p(1)$ for every $j$. Thus, the same joint
limit holds for the vector of componentwise LSEs. This completes the proof
of Theorem~\ref{theorem-lattice}.
\end{proof}

\section{Proof of Theorem \ref{inference-thm} and Proposition \ref{failure-inference}}\label{pivot-limit-section}
We will prove the pivotal limiting distribution for $j=1$ and the proof for $j\neq 1$ follows in a similar fashion.

\subsection{Proof of Theorem \ref{inference-thm}}
\begin{proof}
We prove the case $j=1$. The argument has three steps: prove joint weak convergence of the oracle estimator and the scaled strip widths, translate that convergence into a limit for the block size, and then identify the limiting law by Brownian scaling.

Let $x_0 = (x_{0,1},\ldots, x_{0,d})^T \in \mathcal{X}$ be a design point.

Recall the processes $Q_n$ and $Q$ from
\eqref{two-sided-Qn} and \eqref{two-sided-Q-conv}. For $r\ge 0$ and $\ell>0$,
with the finite-sample arguments restricted to the natural domain of $Q_n$,
define the corresponding two-sided slope functions by
\begin{align}
    \U_n(r,\ell)
    &:=\f{Q_n(r)-Q_n(-\ell)}{r+\ell},
    \label{def-Un-inference}\\
    \U(r,\ell)
    &:=\f{Q(r)-Q(-\ell)}{r+\ell}.
    \label{def-U-inference}
\end{align}
By \eqref{two-sided-Qn} and \eqref{f-oracle},

\[
    T_n:=w_{n,1}^{-1}\big(\fhat_1^{or}(x_{0,1})- f_1^*(x_{0,1})\big)
    = \min_{r \ge 0}\max_{\ell > 0}\U_n(r,\ell).
\]

For brevity, write $(\rhat,\ellhat):=(\rhat_1,\ellhat_1)$ for the scaled
widths introduced after \eqref{f-oracle}.
Applying the finite-sample saddle-point property used in the proof of
Proposition~\ref{small-localization-prop} to $\U_n$ gives
\[
    T_n=\max_{\ell>0}\min_{r\ge0}\U_n(r,\ell)
    =\U_n(\rhat,\ellhat)
\]
and the corresponding optimality inequalities
\[
    \U_n(\rhat,\ell)\le T_n\le\U_n(r,\ellhat),
    \qquad r\ge0,\quad \ell>0.
\]
For $c>0$, define

\[
    H_c := \{0 \le \rhat \le c,\ 0 < \ellhat \le c\}.
\]

By Proposition~\ref{large-localization-prop},
\begin{equation}\label{Hc-localization}
    \lim_{c\to\infty}\limsup_{n\to\infty}\P(H_c^c)=0.
\end{equation}

For fixed $c>0$, let $0\le r_-\le r_+\le c$ and
$0\le\ell_-\le\ell_+\le c$, with $\ell_+>0$, and define
\begin{align*}
    \mathcal V_n(r_-,r_+;\ell_-,\ell_+)
    &:=
    \inf_{r_-\le r\le r_+}
    \sup_{\substack{\ell_-\le\ell\le\ell_+\\\ell>0}}\U_n(r,\ell),\\
    \mathcal V(r_-,r_+;\ell_-,\ell_+)
    &:=
    \inf_{r_-\le r\le r_+}
    \sup_{\substack{\ell_-\le\ell\le\ell_+\\\ell>0}}\U(r,\ell),
\end{align*}
and
\begin{align*}
    T_{n,c}&:=\mathcal V_n(0,c;0,c),
    &T_c&:=\mathcal V(0,c;0,c).
\end{align*}
For compact intervals $I,J\subset[0,c]$, with
$J\cap(0,c]\neq\varnothing$, and $t\in\R$, define
\[
    \Gamma_{n,I,J}(t)
    :=
    \inf_{r\in I}\{Q_n(r)-tr\}
    -
    \inf_{\substack{\ell\in J\\\ell>0}}\{Q_n(-\ell)+t\ell\}.
\]
Let $\Gamma_{I,J}(t)$ be defined analogously with $Q$ in place of $Q_n$.
Writing $I=[r_-,r_+]$ and $J=[\ell_-,\ell_+]$, apply
Lemma~\ref{lm-switch} with $u=r\in I$ and $v=-\ell$, where
$\ell\in J\cap(0,c]$. This gives the following event identities.
\begin{align*}
    \{\mathcal V_n(r_-,r_+;\ell_-,\ell_+)\le t\}
    &=
    \{\Gamma_{n,I,J}(t)\le0\},\\
    \{\mathcal V(r_-,r_+;\ell_-,\ell_+)\le t\}
    &=
    \{\Gamma_{I,J}(t)\le0\}.
\end{align*}
For fixed $0<r_0,\ell_0\le c$ and arbitrary
$(t_1,t_2,t_3)\in\R^3$, apply this event identity to the three
interval pairs
\begin{align*}
    (I_1,J_1)&=([0,r_0],[0,c]),
    &
    \{\mathcal V_n(0,r_0;0,c)\le t_1\}
    &=
    \{\Gamma_{n,I_1,J_1}(t_1)\le0\},\\
    (I_2,J_2)&=([0,c],[0,\ell_0]),
    &
    \{\mathcal V_n(0,c;0,\ell_0)\le t_2\}
    &=
    \{\Gamma_{n,I_2,J_2}(t_2)\le0\},\\
    (I_3,J_3)&=([0,c],[0,c]),
    &
    \{T_{n,c}\le t_3\}
    &=
    \{\Gamma_{n,I_3,J_3}(t_3)\le0\}.
\end{align*}
The corresponding limiting identities hold with $\mathcal V_n$,
$\Gamma_{n,I_k,J_k}$, and $T_{n,c}$ replaced by $\mathcal V$,
$\Gamma_{I_k,J_k}$, and $T_c$, respectively.
The map from $Q_n$ to
$(\Gamma_{n,I_k,J_k}(t_k))_{k=1}^3$ is Lipschitz under the uniform
norm. Consequently, \eqref{two-sided-Q-conv} and the continuous mapping
theorem give
\[
    \big(\Gamma_{n,I_k,J_k}(t_k)\big)_{k=1}^3
    \rightsquigarrow
    \big(\Gamma_{I_k,J_k}(t_k)\big)_{k=1}^3.
\]
By \eqref{two-sided-Q-conv}, $W$ is a two-sided Brownian motion with
independent left and right halves. For each $k$, the two infima defining
$\Gamma_{I_k,J_k}(t_k)$ are independent, and each has a continuous
distribution as the minimum of Brownian motion plus a continuously
differentiable drift on a nondegenerate compact interval
\citep{morters2010brownian}. Their difference therefore has a continuous
distribution, so
\[
    \P\{\Gamma_{I_k,J_k}(t_k)=0\}=0,\qquad k=1,2,3.
\]
Taking intersections of these three event identities, the joint
lower-orthant event is
\[
    \left\{
      \big(\Gamma_{n,I_k,J_k}(t_k)\big)_{k=1}^3
      \in(-\infty,0]^3
    \right\}.
\]
For the limiting vector,
\[
    \left\{
      \big(\Gamma_{I_k,J_k}(t_k)\big)_{k=1}^3
      \in\partial(-\infty,0]^3
    \right\}
    \subseteq
    \bigcup_{k=1}^3\{\Gamma_{I_k,J_k}(t_k)=0\},
\]
and the event on the right has probability zero. Thus $(-\infty,0]^3$ is
a continuity set for the limiting vector. Therefore, for every fixed $c$ and
$r_0,\ell_0>0$, the Portmanteau theorem and the
event identities give
\begin{align*}
    \begin{pmatrix}
        \mathcal V_n(0,r_0;0,c)\\
        \mathcal V_n(0,c;0,\ell_0)\\
        T_{n,c}
    \end{pmatrix}
    \rightsquigarrow
    \begin{pmatrix}
        \mathcal V(0,r_0;0,c)\\
        \mathcal V(0,c;0,\ell_0)\\
        T_c
    \end{pmatrix}.
\end{align*}

For $0<\epsilon<r_0\wedge\ell_0$, applying the same simultaneous switch argument with
$I_1=[r_0-\epsilon,c]$ and
$J_2=[\ell_0-\epsilon,c]$ also yields
\[
    \begin{pmatrix}
        \mathcal V_n(r_0-\epsilon,c;0,c)\\
        \mathcal V_n(0,c;\ell_0-\epsilon,c)\\
        T_{n,c}
    \end{pmatrix}
    \rightsquigarrow
    \begin{pmatrix}
        \mathcal V(r_0-\epsilon,c;0,c)\\
        \mathcal V(0,c;\ell_0-\epsilon,c)\\
        T_c
    \end{pmatrix}.
\]
Let $r_c^*$ and $\ell_c^*$ denote the almost surely unique minimizer and
maximizer in the definition of $T_c$. On the event $H_c$, the
unrestricted finite-sample optimizing pair is feasible for the truncated
problem, so $T_n=T_{n,c}$. Furthermore, the optimality inequalities and
their contrapositives yield
\[
\begin{aligned}
    \rhat\le r_0
        &\Longrightarrow \mathcal V_n(0,r_0;0,c)=T_{n,c},\\
    \ellhat\le \ell_0
        &\Longrightarrow \mathcal V_n(0,c;0,\ell_0)=T_{n,c},\\
    \mathcal V_n(r_0-\epsilon,c;0,c)>T_{n,c}
        &\Longrightarrow \rhat<r_0-\epsilon,\\
    \mathcal V_n(0,c;\ell_0-\epsilon,c)<T_{n,c}
        &\Longrightarrow \ellhat<\ell_0-\epsilon.
\end{aligned}
\]
For the limiting problem, the saddle-point property, continuity on the
compact domains, and uniqueness yield
\[
\begin{aligned}
    r_c^*\le r_0
        &\Longleftrightarrow \mathcal V(0,r_0;0,c)=T_c,\\
    \ell_c^*\le \ell_0
        &\Longleftrightarrow \mathcal V(0,c;0,\ell_0)=T_c,\\
    r_c^*<r_0-\epsilon
        &\Longleftrightarrow
          \mathcal V(r_0-\epsilon,c;0,c)>T_c,\\
    \ell_c^*<\ell_0-\epsilon
        &\Longleftrightarrow
          \mathcal V(0,c;\ell_0-\epsilon,c)<T_c.
\end{aligned}
\]
Define
\[
\begin{aligned}
    E_n
    &:=
      \{T_n\le t,\ \rhat\le r_0,\ \ellhat\le \ell_0\},\\
    F_{n,c}
    &:=
      \{T_{n,c}\le t,\ \mathcal V_n(0,r_0;0,c)=T_{n,c},\
                         \mathcal V_n(0,c;0,\ell_0)=T_{n,c}\},\\
    O_{n,c}
    &:=
      \{T_{n,c}<t-\epsilon,\
        \mathcal V_n(r_0-\epsilon,c;0,c)>T_{n,c},\
        \mathcal V_n(0,c;\ell_0-\epsilon,c)<T_{n,c}\}.
\end{aligned}
\]
The preceding implications yield the finite-sample inclusions
\[
    E_n\cap H_c\subseteq F_{n,c},
    \qquad
    O_{n,c}\cap H_c\subseteq E_n.
\]
Consequently,
\[
    \P(E_n)\le\P(F_{n,c})+\P(H_c^c),
    \qquad
    \P(E_n)\ge\P(O_{n,c})-\P(H_c^c).
\]

The Portmanteau theorem applied to the two localized triples, together with
the limiting equivalences, yields
\begin{align}
    \limsup_{n\to\infty}\P(E_n)
    &\le
      \P\{T_c\le t,\ r_c^*\le r_0,\ \ell_c^*\le \ell_0\}
      +\limsup_{n\to\infty}\P(H_c^c),                 \label{upper-portmanteau-bound}\\
    \liminf_{n\to\infty}\P(E_n)
    &\ge
      \P\{T_c<t-\epsilon,\ r_c^*<r_0-\epsilon,\
                         \ell_c^*<\ell_0-\epsilon\}
      -\limsup_{n\to\infty}\P(H_c^c).                 \label{lower-portmanteau-bound}
\end{align}

Let
\[
    T:=\min_{r\ge0}\max_{\ell>0}\U(r,\ell)
\]
and let $r^*$ and $\ell^*$ be its almost surely unique minimizer and
maximizer respectively. On $\{r^*\le c,\ \ell^*\le c\}$,
uniqueness gives
\[
    (T_c,r_c^*,\ell_c^*)=(T,r^*,\ell^*).
\]
The unrestricted endpoints are finite almost surely, so the probability of
this event tends to one as $c\to\infty$.
In view of \eqref{Hc-localization}, letting $c\to\infty$ and then $\epsilon\downarrow0$ in \eqref{upper-portmanteau-bound} and \eqref{lower-portmanteau-bound} yields the respective limits at every continuity point $(t,r_0,\ell_0)$ of the joint distribution function:
\begin{align}
    \limsup_{n\to\infty}
    \P\{T_n\le t,\rhat\le r_0,\ellhat\le \ell_0\}
    &\le \P\{T\le t,r^*\le r_0,\ell^*\le \ell_0\},\\
    \liminf_{n\to\infty}
    \P\{T_n\le t,\rhat\le r_0,\ellhat\le \ell_0\}
    &\ge \P\{T\le t,r^*\le r_0,\ell^*\le \ell_0\},
\end{align}
Consequently,
\begin{align}\label{joint-conv-Chernoff}
    \begin{pmatrix}
        T_n\\
        \rhat\\
        \ellhat
    \end{pmatrix}
    \rightsquigarrow
    \begin{pmatrix}
        T\\
        r^*\\
        \ell^*
    \end{pmatrix}.
\end{align}

Since $\fhat_1=\fhat_1^{or}-\Ybar$, the two estimators have the same fitted
block. By the half-open convention in \eqref{n_u_v-def} and the equally
spaced product design,
\[
    n_{\rhat,\ellhat}^1
    =(\rhat+\ellhat)w_{n,1}n
    =(\rhat+\ellhat)n^{2/3}.
\]
Moreover,
$\sqrt{n_{\rhat,\ellhat}^1}\,\Ybar
=O_p(n^{1/3})O_p(n^{-1/2})=o_p(1)$, so continuous mapping gives
\begin{align*}
    \sqrt{n_{\rhat,\ellhat}^1}
    \big(\fhat_1(x_{0,1})-f_1^*(x_{0,1})\big)
    &=
    \sqrt{\rhat+\ellhat}\,T_n+o_p(1)\\
    &\rightsquigarrow
    \sqrt{r^*+\ell^*}\,T.
\end{align*}

Let
\[
    s_1 = \left(\f{f_1^{*\prime}(x_{0,1})}{2\sigma}\right)^{2/3}
    \qquad\text{and}\qquad
    s_2 = \left(\f{\sigma^2f_1^{*\prime}(x_{0,1})}{2}\right)^{1/3},
\]
and define $g=s_1 r$ and $h=s_1\ell$. By Brownian scaling,
\begin{align*}
    T
    &= \min_{r \ge 0} \max_{\ell >0}\U(r,\ell)\\
    &= \min_{g \ge 0} \max_{h >0}\f{\sigma s_1 \left(W(g/s_1) - W(-h/s_1)\right)}{g + h}+\f{f_1^{*\prime}(x_{0,1})(g-h)}{2s_1}\\
    &\stackrel{d}{=} \min_{g \ge 0} \max_{h >0}\f{\sigma \sqrt{s_1}\left(W(g) - W(-h)\right)}{g + h}+\f{f_1^{*\prime}(x_{0,1})(g-h)}{2s_1}\\
    &= s_2 \min_{g \ge 0} \max_{h >0}\left\{\f{W(g) - W(-h)}{g + h}+g-h\right\}.
\end{align*}
If $(g^*,h^*)=(s_1 r^*,s_1\ell^*)$, then $s_1(r^*+\ell^*)=g^*+h^*$. Therefore,
\begin{align*}
    \sqrt{r^*+\ell^*}\,T
    &\stackrel{d}{=} \f{s_2}{\sqrt{s_1}}\sqrt{g^*+h^*}\min_{g \ge 0} \max_{h >0}\left\{\f{W(g) - W(-h)}{g + h}+g-h\right\}\\
    &= \sigma \L_1,
\end{align*}
because $s_2/\sqrt{s_1}=\sigma$. This is the claimed pivotal limit.
\end{proof}
\subsection{Proof of Proposition \ref{failure-inference}}
\begin{proof}

We focus on the case $j=1$. The proof follows the argument for Theorem~\ref{inference-thm}, except that we must show that the limiting distribution varies with the local derivative $a:=f_1^{*\prime}(x_{0,1})$. For $r\in\Z_{\ge0}$ and $\ell\in\Z_{\ge1}$, define
\[
    \U_a(r,\ell):=
    \f{\sigma S(r)-\sigma S(-\ell)}{r+\ell}
    +\f{a(r-\ell+1)}{2}.
\]
Then, by \eqref{gcm-rw-def},
\[
    \mathbb{G}_a
    =\inf_{r\in\Z_{\ge0}}\sup_{\ell\in\Z_{\ge1}}\U_a(r,\ell).
\]

We first verify that the saddle pair exists and is almost surely unique. For each fixed $r$, the strong law of large numbers gives $S(-\ell)/\ell\to0$, whence $\U_a(r,\ell)\to-\infty$ as $\ell\to\infty$. Thus, the inner supremum is attained. Moreover,
\[
    \max_{\ell\ge1}\U_a(r,\ell)
    \ge\U_a(r,1)\longrightarrow\infty
    \qquad\text{as }r\to\infty
\]
almost surely, so the outer infimum is also attained. Applying the finite isotonic min--max characterization on a sufficiently large truncation yields a saddle pair on the full lattice $\Z_{\ge0}\times\Z_{\ge1}$.

For any distinct pairs $(r,\ell)\ne(\widetilde r,\widetilde\ell)$, the difference $\U_a(r,\ell)-\U_a(\widetilde r,\widetilde\ell)$ is a non-degenerate Gaussian random variable with non-zero variance:
\[
    \Var\{\U_a(r,\ell)-\U_a(\widetilde r,\widetilde\ell)\}>0.
\]
Consequently, $\P\{\U_a(r,\ell)=\U_a(\widetilde r,\widetilde\ell)\}=0$. Taking a union bound over all countably many pairs shows that the optimizing saddle pair is almost surely unique. We denote it by $(r_a^*,\ell_a^*)$.

Recall the finite-sample slope function $\U_n$ from \eqref{def-Un-inference}, and write $(\rhat,\ellhat):=(\rhat_1,\ellhat_1)$. Under the critical scaling $\beta_1=1/3$, we have $w_{n,1}=n^{-1/3}$, and the integer block widths lie in $\Z_{\ge0}\times\Z_{\ge1}$. As in the proof of Theorem~\ref{inference-thm}, set
\[
    T_n
    :=n^{1/3}\{\fhat_1^{or}(x_{0,1})-f_1^*(x_{0,1})\}
    =\min_{r\ge0}\max_{\ell>0}\U_n(r,\ell),
\]
where the optimization is restricted to the natural domain of $\U_n$.

For a fixed integer $c\ge1$, define the truncated analogues
\[
\begin{aligned}
    T_{n,c}
    &:=n^{1/3}\{\fhat_{1,c}^{or}(x_{0,1})-f_1^*(x_{0,1})\}\\
    &=\min_{r\in[0,c]\cap\Z}
      \max_{\ell\in[1,c]\cap\Z}\U_n(r,\ell),\\
    \mathbb{G}_{a,c}
    &:=\min_{r\in[0,c]\cap\Z}
      \max_{\ell\in[1,c]\cap\Z}\U_a(r,\ell).
\end{aligned}
\]
Let $\rhat_c$ be the largest right width attaining the outer minimum in the definition of $T_{n,c}$. Given $\rhat_c$, let $\ellhat_c$ be the largest left width attaining the inner maximum at $r=\rhat_c$. This deterministic selection rule is used only to resolve finite-sample ties. Let $(r_{a,c}^*,\ell_{a,c}^*)$ be the unique optimizing pair for $\mathbb{G}_{a,c}$.

The finite-dimensional convergence established in the proof for the critical regime of Theorem~\ref{theorem-lattice} gives, for each fixed $c$,
\[
    \bigl(\U_n(r,\ell)\bigr)_{
      \substack{r\in[0,c]\cap\Z\\\ell\in[1,c]\cap\Z}}
    \rightsquigarrow
    \bigl(\U_a(r,\ell)\bigr)_{
      \substack{r\in[0,c]\cap\Z\\\ell\in[1,c]\cap\Z}}.
\]
Because the coordinate mapping from a finite matrix to its saddle value and optimizing indices is continuous on the set of matrices with a unique saddle pair, the continuous mapping theorem implies
\begin{align*}
    \begin{pmatrix}
        T_{n,c}\\
        \rhat_c\\
        \ellhat_c
    \end{pmatrix}
    \rightsquigarrow
    \begin{pmatrix}
        \mathbb{G}_{a,c}\\
        r_{a,c}^*\\
        \ell_{a,c}^*
    \end{pmatrix}.
\end{align*}

On the event $\{\rhat\le c,\ \ellhat\le c\}$, the saddle inequalities imply that $T_{n,c}=T_n$. If the finite collection
$\{\U_n(r,\ell):r\in[0,c]\cap\Z,\ \ell\in[1,c]\cap\Z\}$
has no ties, they also imply that $(\rhat_c,\ellhat_c)=(\rhat,\ellhat)$, since any different truncated optimizing pair would give two distinct pairs with the same value. The finite-dimensional convergence above, the almost sure absence of ties among the limiting Gaussian variables established earlier, and the Portmanteau theorem show that the probability of a finite-sample tie tends to zero for every fixed $c$. Proposition~\ref{large-localization-prop} therefore yields
\[
    \lim_{c\to\infty}\limsup_{n\to\infty}
    \P\{(T_{n,c},\rhat_c,\ellhat_c)
          \ne(T_n,\rhat,\ellhat)\}=0.
\]
Similarly, the almost sure finiteness and uniqueness of $(r_a^*,\ell_a^*)$ imply
\[
    \lim_{c\to\infty}\P\{(\mathbb{G}_{a,c},r_{a,c}^*,\ell_{a,c}^*)
          \ne(\mathbb{G}_a,r_a^*,\ell_a^*)\}=0.
\]
Applying the Lemma~\ref{delocalization-lemma} to these $\R^3$-valued random vectors yields
\begin{align}
    \begin{pmatrix}
        T_n\\
        \rhat\\
        \ellhat
    \end{pmatrix}
    \rightsquigarrow
    \begin{pmatrix}
        \mathbb{G}_a\\
        r_a^*\\
        \ell_a^*
    \end{pmatrix}.                                      \label{critical-joint-conv}
\end{align}

Since $\fhat_1=\fhat_1^{or}-\Ybar$, the two estimators have the same largest fitted constant interval and hence the same scaled widths $(\rhat,\ellhat)$. By the half-open convention in \eqref{n_u_v-def} and the equally spaced product design,
\[
    n_{\rhat,\ellhat}^1=(\rhat+\ellhat)n^{2/3}.
\]
Moreover, $\sqrt{n_{\rhat,\ellhat}^1}\,\Ybar =O_p(n^{1/3})O_p(n^{-1/2})=o_p(1)$, so \eqref{critical-joint-conv} and the continuous mapping theorem give
\begin{align}
    \sqrt{n_{\rhat,\ellhat}^1}
    \{\fhat_1(x_{0,1})-f_1^*(x_{0,1})\}
    &=
    \sqrt{\rhat+\ellhat}\,T_n+o_p(1)\nonumber\\
    &\rightsquigarrow
    \sqrt{r_a^*+\ell_a^*}\,\mathbb{G}_a
    =\L_1^{\mathrm{crit}}(a).                         \label{critical-sqrt-limit}
\end{align}

It remains to show that the family of laws of $\L_1^{\mathrm{crit}}(a)$ varies with $a$. Couple all $a>0$ through the same realization of the random walk $S$. The pair $(r,\ell)=(0,1)$ is the optimizing saddle pair whenever
\[
\begin{aligned}
    \U_a(0,\ell)&\le\U_a(0,1) &&(\ell\ge1),\\
    \U_a(r,1)&\ge\U_a(0,1) &&(r\ge0).
\end{aligned}
\]
For $\ell\ge2$, direct computation shows
\[
    \U_a(0,\ell)-\U_a(0,1)
    =
    \sigma\left\{S(-1)-\f{S(-\ell)}{\ell}\right\}
    -\f{a(\ell-1)}{2}.
\]
Because $S(-\ell)/\ell\to0$ almost surely, $\sup_{\ell\ge2}\frac{\sigma\{S(-1)-S(-\ell)/\ell\}}{\ell-1}<\infty$ a.s., so the first saddle inequality holds simultaneously for all $\ell\ge1$ for all sufficiently large $a$. An identical argument using $S(r)/r\to0$ shows that the second inequality holds for all $r\ge0$ for all sufficiently large $a$. Thus, almost surely, $(0,1)$ is the unique optimizing pair for all sufficiently large $a$. On this event, $\mathbb{G}_a = \U_a(0,1) = -\sigma S(-1) \sim \sigma N(0,1)$, and therefore
\[
    \P\{(r_a^*,\ell_a^*)=(0,1)\}\longrightarrow1
    \qquad(a\to\infty),
\]
which proves
\begin{equation}\label{large-a-normal-limit}
    \L_1^{\mathrm{crit}}(a)/\sigma\rightsquigarrow N(0,1)
    \qquad(a\to\infty).
\end{equation}

For the opposite limit $a\downarrow0$, set
\[
    s_{1,a}=\left(\f{a}{2\sigma}\right)^{2/3},
    \qquad
    s_{2,a}=\sigma\sqrt{s_{1,a}},
\]
so that $a/(2\sigma s_{1,a}^{3/2})=1$. For $g=s_{1,a}r$ and $h=s_{1,a}\ell$, substitution gives
\[
    \f{\U_a(r,\ell)}{s_{2,a}}
    =
    \f{\sqrt{s_{1,a}}\{S(r)-S(-\ell)\}}{g+h}
    +g-h+s_{1,a}.
\]
Define
\[
    T_a:=\f{\mathbb G_a}{s_{2,a}},
    \qquad
    g_a:=s_{1,a}r_a^*,
    \qquad
    h_a:=s_{1,a}\ell_a^*,
\]
and let $\Lambda(t):=W(t)+t^2$, where $W$ is a two-sided Brownian motion. Since $S$ has independent standard Gaussian increments, we have the exact distributional identity
\[
    \bigl(\sqrt{\lambda}S(k)\bigr)_{k\in\Z}
    \stackrel{d}{=}
    \bigl(W(\lambda k)\bigr)_{k\in\Z}
\]
for every $\lambda>0$. Consequently, $(T_a,g_a,h_a)$ has the same joint distribution as the saddle value and optimizing coordinates in
\begin{equation}\label{brownian-lattice-representation}
    \min_{g\in s_{1,a}\Z_{\ge0}}
    \max_{h\in s_{1,a}\Z_{\ge1}}
    \left\{
      \f{\Lambda(g)-\Lambda(-h)}{g+h}+s_{1,a}
    \right\}.
\end{equation}
Equivalently, the value in \eqref{brownian-lattice-representation} is the left slope at the origin of the greatest convex minorant (GCM) of the linear interpolation of $\Lambda(t)+s_{1,a}t$ on $s_{1,a}\Z$. Adding the linear drift $s_{1,a}t$ shifts the resulting GCM slope by $s_{1,a}$ but leaves the contact locations invariant.

As $a\downarrow0$, the lattice mesh $s_{1,a}\to0$, and the piecewise linear interpolant of $\Lambda$ on $s_{1,a}\Z$ converges locally uniformly to $\Lambda$ almost surely. Furthermore, because $W(t)=o(t^2)$ almost surely as $|t|\to\infty$, the quadratic growth in $\Lambda(t)$ ensures that the contact points $(g_a, -h_a)$ of the lattice interpolants and those of $\Lambda$ are eventually confined to a common random compact interval for all sufficiently small $s_{1,a}$. Let $g^*$ and $-h^*$ denote the almost surely unique right and left contact points of the GCM of $\Lambda$ spanning the origin. Under the coupling \eqref{brownian-lattice-representation}, local uniform convergence on this compact interval and the uniqueness of the limiting contact pair $(g^*, -h^*)$ imply the almost-sure convergence of the lattice GCM slope and its contact endpoints. Since this coupling preserves the joint law of $(T_a, g_a, h_a)$, we obtain
\begin{align}
    \begin{pmatrix}
        T_a\\
        g_a\\
        h_a
    \end{pmatrix}
    \rightsquigarrow
    \begin{pmatrix}
        \displaystyle
        \min_{g\ge0}\max_{h>0}
        \left\{\f{W(g)-W(-h)}{g+h}+g-h\right\}\\
        g^*\\
        h^*
    \end{pmatrix}.                                     \label{small-a-saddle-conv}
\end{align}
By the definitions of $T_a$, $g_a$, and $h_a$, the scaling parameters cancel exactly:
\[
    \f{\L_1^{\mathrm{crit}}(a)}{\sigma}
    =\sqrt{\f{g_a+h_a}{s_{1,a}}}\,\f{s_{2,a} T_a}{\sigma}
    =\sqrt{g_a+h_a}\,T_a.
\]
Therefore, \eqref{small-a-saddle-conv} and the continuous mapping theorem yield
\[
    \L_1^{\mathrm{crit}}(a)/\sigma
    \rightsquigarrow
    \sqrt{g^*+h^*}
    \min_{g\ge0}\max_{h>0}
    \left\{\f{W(g)-W(-h)}{g+h}+g-h\right\}
    =:\L_1
    \qquad(a\downarrow0).
\]
The Brownian block pivot $\L_1$ is non-Gaussian \citep{deng2021confidence}. Comparing this limit with the Gaussian limit in \eqref{large-a-normal-limit} demonstrates that the limiting distribution cannot be independent of $a$. This completes the proof.

\end{proof}

\section{Technical lemmas}\label{tech-lemma-sec}

We first record a functional central limit theorem for the process of strip sums.
\begin{lemma}\label{weak-convergence-lemma}
    Suppose $w_{n,1} = n^{-1/3}$, $n_1=n^{\beta_1}$ and $n_{-1}=n/n_1$ where $\beta_1>1/3$. For $m \in \Z_{\ge 0}$, define $S_m = \sum_{i=1}^m \sum_{j=1}^{n_{-1}}\veps_{i,j}$ (with $S_0=0$), where $\{\veps_{i,j}\}$ are i.i.d.\ mean-zero with finite variance $\E \veps_{1,1}^2 = \sigma^2$. For $h \in [0,1]$, define the partial sum process
    \begin{align}\label{Zn-def}
        Z_n(h) = w_{n,1}\left[S_{\lfloor n_1hw_{n,1} \rfloor}+(n_1h w_{n,1} - \lfloor n_1hw_{n,1} \rfloor )\sum_{j=1}^{n_{-1}}\veps_{\lfloor n_1hw_{n,1} \rfloor+1,j}\right].
    \end{align}
    Then $Z_{n} \rightsquigarrow \sigma B$ on $C([0,1])$, where $B$ is a standard Brownian motion.
\end{lemma}
\begin{proof}
    To establish the weak convergence of $Z_n$ on $C([0,1])$, we first prove convergence of its finite-dimensional distributions. We begin with one-dimensional convergence. For fixed $h \in [0,1]$, we have
\begin{align*}
    \Var\left(w_{n,1}S_{\lfloor n_1hw_{n,1} \rfloor}\right)
    &= w_{n,1}^{2} \lfloor n_1hw_{n,1} \rfloor n_{-1}\sigma^2\\
    &\rightarrow h\sigma^2, \quad \text{as} \ n \rightarrow \infty,
\end{align*}
with $w_{n,1} =n^{-1/3}$ and $n=n_1n_{-1}$. Note that we can write $w_{n,1} S_{\lfloor n_1hw_{n,1} \rfloor}$ as $w_{n,1} \sum_{i=1}^{m_n}\veps_i$ where $m_n = \lfloor n_1hw_{n,1} \rfloor n_{-1}$ and $\veps_i \stackrel{iid}{\sim} P_{\veps}$, where $P_{\veps}$ is the distribution of $\veps$. Therefore, for $\eta>0$,
\begin{align*}
	    w_{n,1}^2 \E \sum_{i=1}^{m_n} \veps_i^2 I\left\{|\veps_i| \ge w_{n,1}^{-1}\eta \right\} &= w_{n,1}^2 \lfloor n_1hw_{n,1} \rfloor n_{-1} \E \veps_1^2 I\left\{ |\veps_1| \ge w_{n,1}^{-1}\eta \right\}\\
	    &\le h \E \veps_1^2 I\left\{ |\veps_1| \ge w_{n,1}^{-1}\eta \right\},
\end{align*}
which converges to $0$ by the dominated convergence theorem. Therefore, by the Lindeberg--Feller central limit theorem, we have
\begin{align}
    w_{n,1}S_{\lfloor n_1hw_{n,1} \rfloor} \rightsquigarrow N(0, h\sigma^2).\label{eqn-27}
\end{align}
By Chebyshev's inequality and $\beta_1> 1/3$, for any $y>0$,
\begin{align}
    \P\left\{w_{n,1}\left|(n_1h w_{n,1} - \lfloor n_1hw_{n,1} \rfloor )\sum_{j=1}^{n_{-1}}\veps_{\lfloor n_1hw_{n,1} \rfloor+1,j}\right|\ge y\right\} \le \f{w_{n,1}^{2}\sigma^2 n_{-1}}{y^2},\label{eqn-28}
\end{align}
which converges to $0$ as $n \rightarrow \infty$. Together, \eqref{Zn-def}--\eqref{eqn-28} and Slutsky's theorem give
\[
    Z_n(h)\rightsquigarrow N(0,h\sigma^2).
\]

Now fix $0\le h_1<\cdots<h_k\le1$. By \eqref{eqn-28} and a union bound, the interpolation terms in $Z_n(h_1),\ldots,Z_n(h_k)$ are jointly $o_p(1)$. Moreover, for $1\le p,q\le k$,
\begin{align*}
    &\operatorname{Cov}\left(
        w_{n,1}S_{\lfloor n_1h_pw_{n,1}\rfloor},
        w_{n,1}S_{\lfloor n_1h_qw_{n,1}\rfloor}
    \right)\\
    &\qquad
    =w_{n,1}^2n_{-1}\sigma^2
      \min\left\{
          \lfloor n_1h_pw_{n,1}\rfloor,
          \lfloor n_1h_qw_{n,1}\rfloor
      \right\}
    \longrightarrow \sigma^2\min\{h_p,h_q\}.
\end{align*}
The same Lindeberg calculation as above applies to every linear combination of these partial sums. Hence, the Cram\'er--Wold device and Slutsky's theorem yield
\[
    \bigl(Z_n(h_1),\ldots,Z_n(h_k)\bigr)
    \rightsquigarrow
    \sigma\bigl(B(h_1),\ldots,B(h_k)\bigr).
\]
Thus, the finite-dimensional distributions of $Z_n$ converge to those of $\sigma B$.

It remains to establish tightness. By Prokhorov's theorem and the Arzel\`a--Ascoli theorem (see Theorems 7.1 and 7.3 in \citet{billingsley1999convergence}), it suffices to prove that for any $\eta>0$,
\begin{align}\label{eqn-29}
    \lim_{\delta \rightarrow 0}\limsup_{n \rightarrow \infty}\P\left\{  \sup_{|s-t| \le \delta} |Z_n(s)-Z_n(t)| \ge \eta\right\} = 0.
\end{align}

Let $m_n=n_1w_{n,1}$, $q_n=\lceil m_n\delta\rceil$, and
$v_n=\lceil m_n/q_n\rceil$. Define
\[
    t_i=\left(\f{iq_n}{m_n}\right)\wedge1,
    \qquad i=0,\ldots,v_n.
\]
Then $v_n\le\delta^{-1}+1$. Since $q_n/m_n\ge\delta$, any two points
at distance at most $\delta$ lie in the same block or in two adjacent blocks.
The piecewise linearity of $Z_n$ therefore gives
\[
    \sup_{|s-t|\le\delta}|Z_n(s)-Z_n(t)|
    \le
    3\max_{1\le i\le v_n}
    \sup_{t_{i-1}\le s\le t_i}
    |Z_n(s)-Z_n(t_{i-1})|.
\]

For a full block, define
\[
    Y_{n,i}
    :=w_{n,1}
      \sum_{r=(i-1)q_n+1}^{iq_n}
      \sum_{j=1}^{n_{-1}}\veps_{r,j}.
\]
The last block may be enlarged to a full block by adding unused independent
increments. Since
\[
    \Var(Y_{n,i})
    =\f{q_n}{m_n}\sigma^2
    \longrightarrow\delta\sigma^2,
\]
the Lindeberg calculation used for \eqref{eqn-27} gives
$Y_{n,i}\rightsquigarrow N(0,\delta\sigma^2)$ for fixed $\delta$.
For $\delta$ small enough that $36\delta\sigma^2<\eta^2$, the
L\'evy--Ottaviani inequality and Chebyshev's inequality imply
\begin{align*}
    &\limsup_{n\to\infty}
      \P\left\{\sup_{|s-t|\le\delta}|Z_n(s)-Z_n(t)|\ge\eta\right\}\\
    &\quad\le
      \f{\delta^{-1}+1}
      {1-36\delta\sigma^2/\eta^2}
      \P\left\{|\sigma\sqrt{\delta}Z|\ge\eta/6\right\},
\end{align*}
where $Z\sim N(0,1)$. The right-hand side converges to zero as
$\delta\downarrow0$. Indeed,
\[
    \delta^{-1}
    \P\left\{|Z|\ge\f{\eta}{6\sigma\sqrt{\delta}}\right\}
    \le
    \f{36\sigma^2}{\eta^2}
    \E\left[
       Z^2\ind\left\{|Z|\ge\f{\eta}{6\sigma\sqrt{\delta}}\right\}
    \right]
    \longrightarrow0.
\]
This proves \eqref{eqn-29}.

\end{proof}

The following lemma is used to control the noise term of LSEs.
For an interval $I\subset\R$, define
\[
    \vepsbar\big|_{I}
    := \f{\sum_{i=1}^n \veps_i\,\ind\{x_{i,1}\in I\}}{\sum_{i=1}^n \ind\{x_{i,1}\in I\}}.
\]
\begin{lemma}\label{error-control}
    Under the same conditions as in Theorem \ref{theorem-lattice}, set
    $d_n=(n_1w_{n,1})^{-1}$. For $r\ge0$, define
    \[
        \widetilde N_r
        :=\sum_{i=1}^n
        \ind\{x_{i,1}\in(x_{0,1}-d_nw_{n,1},
                          x_{0,1}+rw_{n,1}]\}.
    \]
    Then
    \[
        \sup_{\ell\ge d_n}
        \left|\vepsbar\big|_{(x_{0,1}-\ell w_{n,1},
                                  x_{0,1}+rw_{n,1}]}\right|
        =O_p\big(\sigma \widetilde N_r^{-1/2}\big),
    \]
    where the $O_p$ constant does not depend on $r$. The bound is
    $O_p(w_{n,1})$ whenever $\widetilde N_r\gtrsim w_{n,1}^{-2}$. This holds
    for $r=0$ when $\beta_1\le1/3$ and for $r\ge1$ when $\beta_1>1/3$.
    The same conclusions hold after interchanging the two endpoints.
\end{lemma}
\begin{proof}
    Fix $r\ge0$ and write
    $I_\ell=(x_{0,1}-\ell w_{n,1},x_{0,1}+rw_{n,1}]$ for
    $\ell\ge d_n$. Let
    $M_\ell=\sum_{i=1}^n\veps_i\ind\{x_{i,1}\in I_\ell\}$ and
    $\mathcal G_\ell=\sigma(\veps_i\ind\{x_{i,1}\in I_\ell\}:i=1,\ldots,n)$.
    Then $\{M_\ell\}_{\ell\ge d_n}$ is a martingale with respect to
    $\{\mathcal G_\ell\}_{\ell\ge d_n}$. If
    $m_\ell=\sum_{i=1}^n\ind\{x_{i,1}\in I_\ell\}$, then
    $\E M_\ell^2=\sigma^2m_\ell$ and
    $m_\ell\ge m_{d_n}=\widetilde N_r$.

    Partition $[d_n,\infty)$ into
    $B_j=\{\ell\ge d_n:2^{j-1}\le m_\ell\le2^j\}$ for
    $j\ge j_0:=\lceil\log_2\widetilde N_r\rceil$. On $B_j$, the denominator
    of $\vepsbar|_{I_\ell}=M_\ell/m_\ell$ is at least $2^{j-1}$, and Doob's
    inequality gives
    $\E\sup_{\{\ell:m_\ell\le2^j\}}M_\ell^2\le4\sigma^22^j$. Therefore,
    \begin{align*}
        \E\sup_{\ell\ge d_n}\left|\vepsbar|_{I_\ell}\right|
        \le \sum_{j \ge j_0} 2^{1-j}\Big(\E \sup_{\{\ell :\, m_\ell \le 2^j\}} M_\ell^2\Big)^{1/2}
        \le \sum_{j \ge j_0} 2^{1-j}\cdot 2\sigma 2^{j/2}
        = 4\sigma \sum_{j \ge j_0} 2^{-j/2}
        \le C \sigma \widetilde N_r^{-1/2},
    \end{align*}
    where $C$ is an absolute constant. Markov's inequality proves
    the claim. Interchanging the endpoints gives the analogous bound for a
    fixed left extension.
\end{proof}

\begin{lemma}\label{lm-switch}
    Let $F$ be a continuous stochastic process on $[-M,M]$, where $M>0$ is a constant, let $I\subset[0,M]$ be a nonempty compact set, and let $J\subset[-M,0)$ be nonempty. Then, for every $t\in \R$,
        \begin{align*}
            \left\{ \inf_{u \in I} \sup_{v \in J} \f{F(u)- F(v)}{u-v}\le t\right\}=\left\{ \inf_{u \in I}\sup_{v \in J} \bigl[F(u) - F(v) - t(u-v)\bigr] \le 0 \right\}.
        \end{align*}
    In particular, the two events have the same probability.
    \end{lemma}
    \begin{proof}
    Fix $t \in \R$ and define
    \begin{align*}
        A &= \inf_{u \in I}\sup_{v \in J} \f{F(u)- F(v)}{u-v},\\
        B &= \inf_{u \in I} \phi(u), \qquad \text{where } \phi(u) := \sup_{v \in J} \bigl[F(u) - F(v) - t(u-v)\bigr].
    \end{align*}
    Since $F$ is continuous on the compact interval $[-M,M]$, $\phi$ is finite, and $|\phi(u) - \phi(u')| \le |F(u) - F(u')| + |t|\,|u-u'|$ for $u, u' \in I$, so $\phi$ is continuous on $I$. As $I$ is compact, there exists $u_* \in I$ with $B = \phi(u_*)$. Note also that $0 < u - v \le 2M$ for all $u \in I$ and $v \in J$.

    Suppose first that $A \le t$. For every $\epsilon>0$ there exists $u_\epsilon \in I$ such that $\f{F(u_\epsilon)-F(v)}{u_\epsilon-v} \le t+\epsilon$ for all $v \in J$. Multiplying by $u_\epsilon - v \in (0, 2M]$ gives $F(u_\epsilon)-F(v) - t(u_\epsilon - v) \le 2M\epsilon$ for all $v \in J$, so $B \le \phi(u_\epsilon) \le 2M\epsilon$. Letting $\epsilon \downarrow 0$ yields $B \le 0$.

    Conversely, suppose that $B \le 0$. Then $\phi(u_*) \le 0$, that is, $F(u_*)- F(v) - t(u_*-v) \le 0$ for all $v \in J$. Dividing by $u_* - v>0$ gives $\f{F(u_*) - F(v)}{u_*-v}\le t$ for all $v \in J$, and hence $A \le \sup_{v \in J}\f{F(u_*) - F(v)}{u_*-v} \le t$.
    \end{proof}

We state the following converging-together result from
\citet{rao1969estimation}.
    \begin{lemma}\label{delocalization-lemma}
        Let $\{W_{n,c}\}$, $\{W_n\}$, $\{W_c\}$, and $W$ be random elements
        taking values in a common metric space. Suppose that:
\begin{itemize}
    \item $\lim_{c \rightarrow  \infty}\limsup_{n \rightarrow \infty}\P\{W_{n,c} \neq W_n\} = 0$.
    \item  For every $c > 0$, $W_{n,c} \rightsquigarrow W_c$ as $n \rightarrow  \infty$.
    \item $\lim_{c \rightarrow \infty} \P\{W_c \neq W\} = 0$.
\end{itemize}
Then $W_n \rightsquigarrow W$ as $n \rightarrow \infty$.
\end{lemma}

\bibliographystyle{abbrvnat}
\bibliography{references}

\end{document}